\documentclass[leqno,11pt]{amsart}
\usepackage{amsmath,amstext,amssymb,amsopn,amsthm,mathrsfs,color}
\usepackage{enumerate}
\usepackage{graphicx}
\usepackage{multirow}
\usepackage{float}

\allowdisplaybreaks

\DeclareMathOperator*{\essup}{ess\,sup}

\DeclareMathOperator{\sgn}{sgn}

\newtheorem{thm}{Theorem}[section]
\newtheorem*{thm*}{Theorem}

\newtheorem{propo}[thm]{Proposition}
\newtheorem{lem}[thm]{Lemma}

\newtheorem*{rem*}{Remark}
\newtheorem{rem}[thm]{Remark}

\newcommand{\R}{\mathbb{R}}

\def\a{\alpha}

\def\s{\sigma}

\begin{document}

\author[A. Nowak]{Adam Nowak}
\address{Adam Nowak \newline
Instytut Matematyczny,
Polska Akademia Nauk\newline
\'Sniadeckich 8,
00-956 Warszawa, Poland}
\email{adam.nowak@impan.pl}

\author[K. Stempak]{Krzysztof Stempak} 
\address{Krzysztof Stempak     \newline
      Instytut Matematyki i Informatyki,
      Politechnika Wroc\l{}awska       \newline
      Wyb{.} Wyspia\'nskiego 27,
      50--370 Wroc\l{}aw, Poland}
\email{krzysztof.stempak@pwr.wroc.pl}


\footnotetext{
\indent This research was supported by a grant from the National Science Centre of Poland.
}

\title[Potential operators]
	{Sharp estimates for potential operators \\associated with Laguerre and Dunkl-Laguerre expansions} 

\begin{abstract}
We study potential operators associated with Laguerre function expansions of convolution and Hermite types,
and with Dunkl-Laguerre expansions. We prove qualitatively sharp estimates of the corresponding
potential kernels. Then we characterize those $1\le p,q\le\infty$, 
for which the potential operators are $L^p-L^q$ bounded. These results are sharp analogues of the 
classical Hardy-Littlewood-Sobolev fractional integration theorem in the Laguerre and Dunkl-Laguerre settings.
\end{abstract}

\maketitle

\section{Introduction}
In recent years the study of potential theory for `Laplacians' associated with classical 
orthogonal expansions attracted considerable attention. 
The model case of the Riesz potentials $I^\s=(-\Delta)^{-\s}$, where
$\Delta$ denotes the Euclidean Laplacian in $\R^d$, $d\ge1$, is related to continuous expansions 
with respect to the system $\{\exp(2\pi i\langle \cdot,\xi\rangle) : \xi \in \R^d\}$. 
The $L^p-L^q$ boundedness of $I^\s$, $0<\s<d/2$, is characterized by the celebrated
Hardy-Littlewood-Sobolev theorem.

In \cite{BoTo1} Bongioanni and Torrea investigated potential operators related to the harmonic oscillator
$\mathcal H=-\Delta+\|x\|^2$, which plays the role of a Laplacian in the context of multi-dimensional 
Hermite function expansions. Some complementary comments on that research are contained in 
\cite[Section 2]{NoSt1}. More recently, in \cite{NoSt2} it was shown that the
$L^p-L^q$ bounds obtained in \cite{BoTo1} for the potential operator $\mathcal I^\s=\mathcal H^{-\s}$
are in fact sharp in the sense of admissible $p$ and $q$. 
This was achieved as a consequence of qualitatively sharp estimates for the integral
kernel of $\mathcal I^\s$ established also in \cite{NoSt2}.
A thorough study of potential operators associated with classical one-dimensional 
Jacobi and Fourier-Bessel expansions has just been furnished by Nowak and Roncal \cite{NR}. 
The corresponding analysis is based on sharp estimates of the potential kernels proved in that work.

Potential operators related to multi-dimensional Laguerre operators in $\R^d_+$, and to 
the Dunkl harmonic oscillator in $\R^d$ with the underlying reflection group isomorphic to $\mathbb Z^d_2$,
were investigated by the authors in \cite{NoSt1}.
Recall that the latter `Laplacian' is a differential-difference operator, and its eigenfunctions express via
certain Laguerre functions. Hence the associated expansions are sometimes referred to as
Dunkl-Laguerre expansions. 
The aim of \cite{NoSt1} was to prove $L^p-L^q$ bounds for the considered potential operators for a possibly
wide range of $p$ and $q$. Another objective was to obtain in a similar spirit 
two-weight $L^p-L^q$ bounds, with power weights involved. All these results in \cite{NoSt1} were derived
as indirect and somewhat tricky consequences of analogous theory for $\mathcal{I}^{\s}$, and under the
restriction $\a \in [-1/2,\infty)^d$ on the multi-parameter of type.

The present paper is motivated by the natural question to what extent the results of \cite{NoSt1} are
optimal in the sense of admissible $p$ and $q$. Further motivation comes from a related problem, but
certainly of independent interest, of describing the behavior of the relevant potential kernels via
pointwise estimates. Finally, yet another motivation follows from a desire
to get rid of the above mentioned restriction on $\a$.
All these inspirations found a positive outcome.
For technical reasons, we consider only $d=1$ and thus work in dimension one, 
otherwise the analysis we present
would become much more sophisticated. Then we investigate the settings from \cite{NoSt1}, that is,
according to the terminology used in \cite{Th},
the situations of Laguerre function expansions of convolution and Hermite types, 
and Dunkl-Laguerre expansions
(see Section \ref{sec:prel} for the definitions), with no artificial restrictions on $\a$ imposed.
We prove qualitatively sharp estimates for the relevant potential kernels 
(Theorems \ref{thm:conv_ker} and \ref{thm:d_ker}). 
Then we characterize those $1\le p,q \le \infty$, for which
the potential operators are $L^p-L^q$ bounded 
(Theorems \ref{thm:LpLqlag}, \ref{thm:lag_H} and \ref{thm:lag_D}).
In particular, it follows that the unweighted $L^p-L^q$ bounds from \cite{NoSt1} are in fact sharp,
at least in the one-dimensional case.

It is remarkable that our present results enable further research which is no doubt of interest, but
beyond the scope of this work. Let us mention here the following issues:
\begin{itemize}
\item characterization of weak type and restricted weak type inequalities for the potential operators
	(see \cite[Theorems 2.3, 2.4, 2.7, 2.8]{NR}),
\item characterization of two-weight $L^p-L^q$ inequalities for the potential operators, with power
	weights involved (see \cite[Theorems 1.2, 2.5, 3.3, 4.2, 6.2]{NoSt1}; note that it is known that
	at least some of these results are not optimal),
\item development of analogous theory for other variants of fractional integrals in the Laguerre and
	Dunkl-Laguerre settings, or more generally, for Laplace-Stieltjes type multipliers (see the comments
	closing \cite[Section 2]{NoSt1} and \cite[Section 3]{NoSt1}).
\end{itemize}

Finally, we remark that although the present framework is one-dimensional, it has, 
at least in the setting of Laguerre expansions of convolution type, 
a multi-dimensional background. More precisely, if $\a=n-1$, $n\ge1$, 
then the context of Laguerre function expansions of convolution type is related to 
a `radial' analysis in $\mathbb C^n$ equipped with twisted convolution, see \cite{St,Th} for details. 
Continuing this line of thought, we note that the system of Laguerre functions of Hermite type also
has a multi-dimensional connection, since it consists of eigenfunctions of the Hankel transform.

The paper is organized as follows. In Section \ref{sec:prel} we briefly introduce the settings to be
investigated and state the main results 
(Theorems \ref{thm:conv_ker}-\ref{thm:d_ker} and Theorem \ref{thm:lag_D}). The corresponding proofs
are contained in the two succeeding sections. In Section \ref{sec:pot} we show qualitatively sharp
estimates for the relevant potential kernels. Section \ref{sec:LpLq} is devoted to characterizing
$L^p-L^q$ boundedness of the Laguerre and Dunkl-Laguerre potential operators.

Throughout the paper we use a standard notation, which is consistent with that used in \cite{NoSt1,NoSt2}. 
In particular, we write $X\lesssim Y$ to indicate that $X\leq CY$ 
with a positive constant $C$ independent of significant quantities. We shall write $X \simeq Y$ when
simultaneously $X \lesssim Y$ and $Y \lesssim X$. Furthermore, 
$X\simeq\simeq  Y\exp(-cZ) $ means that there exist positive constants $C, c_1, c_2$, 
independent of significant quantities, such that
$$
C^{-1}Y\exp(-c_1Z)\le X\le C\,Y\exp(-c_2Z).
$$
In a number of places we will use natural and self-explanatory generalizations
of the ``$\simeq \simeq$'' relation, for instance, in connection with certain integrals
involving exponential factors. In such cases the exact meaning will be clear from the context.
By convention, ``$\simeq \simeq$'' is understood as ``$\simeq$'' whenever  no exponential
factors are involved.

We treat positive kernels and integrals as expressions valued in the extended half-line $[0,\infty]$.
Similar remark concerns expressions occurring in various estimates, with the natural limiting interpretations
like, for instance, $(0^+)^{\beta} = \infty$ when $\beta < 0$.

\section{Preliminaries and statement of results} \label{sec:prel}

We will consider two interrelated settings corresponding to one-dimensional Laguerre function expansions
of convolution type and of Hermite type. Also, we will study the one-dimensional 
context of Dunkl-Laguerre expansions
associated with the Dunkl harmonic oscillator and the underlying group of reflections isomorphic to
$\mathbb{Z}_2$. The latter situation may be regarded as an extension of that of Laguerre function expansions
of convolution type, see Section \ref{ssec:dunkl} below. All the three frameworks in question have deep
roots in the existing literature. In particular, in the last decade they were widely investigated
from the harmonic analysis perspective. For all the facts (tacitly)
invoked in what follows we refer to \cite{NoSt1} and references given there.
\subsection{Laguerre function setting of convolution type} \label{ssec:conv}
Let $\a> -1$. The Laguerre functions of convolution type are given by
$$
\ell_n^{\a}(x) = c_n^{\a}\, L_n^{\a}\big(x^2\big) \exp\big(-x^2/2\big), \qquad x > 0,
$$
where $c_n^{\a}>0$ are the normalizing constants, and $L_n^{\a}$, $n \ge 0$, are the classical Laguerre
polynomials. The system $\{\ell_n^{\a}:n \ge 0\}$ is an orthonormal basis in $L^2(d\mu_{\a})$, where
$\mu_{\a}$ is the measure on the half-line $\mathbb{R}_{+}=(0,\infty)$ defined by
$$
d\mu_{\a}(x) = x^{2\a+1}\, dx.
$$
The $\ell_n^{\a}$ are eigenfunctions of the Laguerre `Laplacian'
$$
L_{\a} = - \frac{d^2}{dx^2} -\frac{2\a+1}{x}\frac{d}{dx} + x^2,
$$
we have $L_{\a}\ell_n^{\a} = (4n+2\a+2)\ell_n^{\a}$. We denote by the same symbol $L_{\a}$ the
natural self-adjoint extension whose spectral resolution is given by the $\ell_n^{\a}$.
The integral kernel $G_t^{\a}(x,y)$ of the Laguerre heat semigroup $\{\exp(-tL_{\a})\}$ can be expressed
explicitly in terms of the modified Bessel function $I_{\a}$. More precisely,
\begin{equation} \label{hkl}
G_t^{\a}(x,y) = \frac{1}{\sinh 2t} \exp\Big( -\frac{1}2 \coth(2t) \big(x^2+y^2\big) \Big) (xy)^{-\a}
	I_{\a}\Big( \frac{xy}{\sinh 2t}\Big), \qquad x,y >0.
\end{equation}

Given $\s>0$, we consider the potential operator
$$
I^{\a,\s}f(x) = \int_0^{\infty} K^{\a,\s}(x,y)f(y)\, d\mu_{\a}(y), \qquad x > 0,
$$
where the potential kernel is defined as
$$
K^{\a,\s}(x,y) = \frac{1}{\Gamma(\s)} \int_0^{\infty} G_t^{\a}(x,y)\, t^{\s-1}\, dt, \qquad x,y > 0.
$$
We will prove the following general and qualitatively sharp estimates of $K^{\a,\s}(x,y)$. 
\begin{thm} \label{thm:conv_ker}
Let $\a>-1$ and let $\s>0$. The following estimates hold uniformly in $x,y>0$.
\begin{itemize}
\item[(i)] If $x+y \le 1$, then
$$
K^{\a,\s}(x,y) \simeq  \chi_{\{\s>\a+1\}} + \chi_{\{\s=\a+1\}} \log\frac{1}{x+y} 
	 + (x+y)^{-2\a-1} 
	\begin{cases}
		|x-y|^{2\s-1}, & \s< 1/2,\\
		1+\log \frac{x+y}{|x-y|}, & \s=1/2,\\
		(x+y)^{2\s-1}, & \s>1/2.
	\end{cases}
$$
\item[(ii)] If $x+y > 1$, then
$$
K^{\a,\s}(x,y) \simeq \simeq (x+y)^{-2\a-1} \exp\big(-c|x-y|(x+y)\big) 
	\begin{cases}
		|x-y|^{2\s-1}, & \s< 1/2,\\
		1+\log^+ \frac{1}{|x-y|(x+y)}, & \s=1/2,\\
		(x+y)^{1-2\s}, & \s>1/2.
	\end{cases}
$$
\end{itemize}
\end{thm}
Thus, among other things, we see that the kernel behaves
in an essentially different way, depending on whether $(x,y)$ is close to the origin of
$\mathbb{R}^2$ or far from it. We remark that
under the restrictions $\a \ge -1/2$ and $\s < \a+1$,
an upper bound for $K^{\a,\s}(x,y)$ was obtained recently in \cite[Proposition 5.1]{CiRo}.
 
The description of $K^{\a,\s}(x,y)$ from Theorem \ref{thm:conv_ker} 
enables a direct analysis of the potential operator $I^{\a,\s}$.
In particular, it allows us to characterize those $1\le p,q \le \infty$, for which $I^{\a,\s}$ is
$L^p-L^q$ bounded, see also Figure \ref{fig1} below.
\begin{thm} \label{thm:LpLqlag}
Let $\a > -1$, $\s>0$ and $1\le p,q \le \infty$.
\begin{itemize}
\item[(a)] If $\a \ge -1/2$, then $I^{\a,\s}$ is bounded from $L^p(d\mu_{\a})$ to $L^q(d\mu_{\a})$
if and only if
$$
\frac{1}{p} - \frac{\s}{\a+1} \le \frac{1}q < \frac{1}{p} + \frac{\s}{\a+1} \quad \textrm{and} \quad
\bigg(\frac{1}p,\frac1{q}\bigg) \notin 
	\bigg\{ \Big(\frac{\s}{\a+1},0\Big),\Big(1,1-\frac{\s}{\a+1}\Big)\bigg\}.
$$
\item[(b)] If $\a < -1/2$, then $I^{\a,\s}$ is bounded from $L^p(d\mu_{\a})$ to $L^q(d\mu_{\a})$
if and only if
$$
\frac{1}p + \frac{\s}{\a} \le \frac{1}q < \frac{1}p + \frac{\s}{\a+1}.
$$
\end{itemize}
\end{thm}
Note that the sufficiency part of Theorem \ref{thm:LpLqlag} (a) was essentially known to the authors earlier, 
even in the multi-dimensional case, see \cite[Theorem 4.1]{NoSt1}. Here, however, we give
a direct proof which offers a better insight into the structure of $I^{\a,\s}$. The necessity part, 
as well as item (b) in Theorem~\ref{thm:LpLqlag}, is new. 
It seems a bit surprising that the conditions in (a) and (b) are different, since many known
results related to the system $\{\ell_n^{\a}\}$ are homogeneous in $\a > -1$, without any `phase shift' at
$\a=-1/2$; see for instance \cite[Theorem 4.1]{NoSz} and \cite[Corollary 4.2]{NoSz}.
\begin{figure} [ht]
\includegraphics[width=0.9\textwidth]{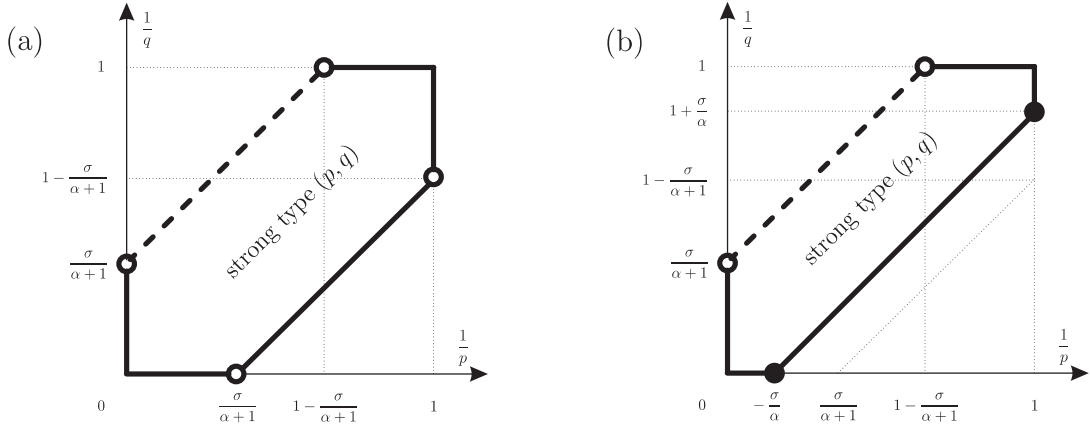}
\caption{Optimal sets of $\big(\frac{1}p,\frac{1}q\big)$ for which ${I}^{\a,\s}$ is $L^p-L^q$ bounded when 
	$\s < \alpha+1$; 
	(a) the case of $\a \ge -1/2$, (b) the case of $\a<-1/2$.}
 \label{fig1}
\end{figure}

\subsection{Laguerre function setting of Hermite type} \label{ssec:hermite}
This Laguerre context is derived from the previous one by modifying the Laguerre functions $\ell_n^{\a}$
so as to make the resulting system orthonormal with respect to Lebesgue measure $dx$ in $\mathbb{R}_+$.
Thus, given a parameter $\a > -1$, we consider the functions
$$
\varphi_n^{\a}(x) = x^{\a+1/2}\ell_n^{\a}(x), \qquad x > 0.
$$
Then the system $\{\varphi_n^{\a}:n\ge 0\}$ is an orthonormal basis in $L^2(dx)$. The associated
`Laplacian' is 
$$
L_{\a}^{H} = - \frac{d^2}{dx^2} + x^2 + \frac{(\a-1/2)(\a+1/2)}{x^2},
$$
and we have $L_{\a}^{H}\varphi_n^{\a} = (4n+2\a+2)\varphi_n^{\a}$.
The Laguerre heat semigroup $\{\exp(-tL^{H}_{\a})\}$, generated by means of the natural self-adjoint
extension of $L_H^{\a}$ in this context, has an integral representation. The associated heat kernel
is $(xy)^{\a+1/2}G^{\a}_t(x,y)$, $x,y>0$, see \eqref{hkl}.

For $\s>0$, consider the potential operator
$$
I_{H}^{\a,\s}f(x) = \int_0^{\infty} K_{H}^{\a,\s}(x,y)f(y)\, dy, \qquad x > 0,
$$
where
\begin{equation} \label{linkKH}
K_{H}^{\a,\s}(x,y) = (xy)^{\a+1/2} K^{\a,\s}(x,y), \qquad x,y > 0.
\end{equation}
Because of this simple link between the two potential kernels, Theorem \ref{thm:conv_ker} gives
qualitatively sharp estimates also for $K^{\a,\s}_{H}(x,y)$.
Then, taking into account the behavior of the kernel for $\a < -1/2$ and $x$ and $y$ small,
it is not hard to see that $I_{H}^{\a,\s}$ can be well defined on $L^p$ only if $\frac{1}p < \a +3/2$.
In fact, when  $\a < -1/2$ we have $L^p(dx)\subset {\rm Dom}\, I_{H}^{\a,\s}$ if and only if
$p\in(2/(2\a+3),\infty]$ (here ${\rm Dom}\, I_{H}^{\a,\s}$ denotes the natural domain of the integral
operator $I_{H}^{\a,\s}$). Note that this case is qualitatively different from the case $\a\ge-1/2$,
see the comments preceding \cite[Theorem 3.1]{NoSt1}. 
This restriction on $p$ together with a restriction on $q$
in Theorem \ref{thm:lag_H} (b) below is an instance of the so-called \emph{pencil phenomenon}, see
e.g.\ \cite{NoSj}.

The following result gives a complete and sharp description of $L^p-L^q$ boundedness of
$I^{\a,\s}_H$. It reveals that for $\a \ge -1/2$, $I_H^{\a,\s}$ behaves exactly like $I^{-1/2,\s}$,
see Figure \ref{fig1} (a) with $\alpha = -1/2$, and thus like the Hermite potential operator
$\mathcal{I}^{\s}$ (see the case of $I_D^{-1/2,\s}$ in Theorem \ref{thm:lag_D} below).
On the other hand, for $\a < -1/2$ the $L^p-L^q$
behavior of $I_H^{\a,\s}$ is more subtle, 
and partially this is caused by the restriction on $p$ mentioned above.
In particular, the region characterizing those
$(\frac{1}p,\frac{1}q)$ for which $I_H^{\a,\s}$ is $L^p-L^q$ bounded may take various peculiar shapes,
see Figure \ref{fig2} below.

\begin{thm} \label{thm:lag_H}
Let $\a > -1$, $\s>0$ and $1 \le p,q \le \infty$.  
\begin{itemize}
\item[(a)] If $\a \ge -1/2$, then ${I}^{\a,\s}_{H}$ is bounded from $L^p(dx)$ to $L^q(dx)$ if and only if
$$
\frac{1}p - 2\s \le \frac{1}q < \frac{1}p + 2\s \quad \textrm{and} \quad
	\bigg(\frac{1}p,\frac{1}q\bigg) \notin \big\{(2\s,0), (1,1-2\s)\big\}.
$$
\item[(b)] If $\a< -1/2$ and $p > 2/(2\a+3)$, 
then ${I}^{\a,\s}_{H}$ is bounded from $L^p(dx)$ to $L^q(dx)$ if and only if
$$
\frac{1}p - 2\s \le \frac{1}q < \frac{1}p + 2\s \quad  \textrm{and} \quad \frac{1}q > -\a-\frac{1}2.
$$
\end{itemize}
\end{thm}
Note that the sufficiency part of Theorem \ref{thm:lag_H} (a) was known earlier, 
even in the multi-dimensional case, see \cite[Theorem 3.1]{NoSt1}.
Apart from that, the result is new.
\begin{figure} [ht]
\includegraphics[width=0.9\textwidth]{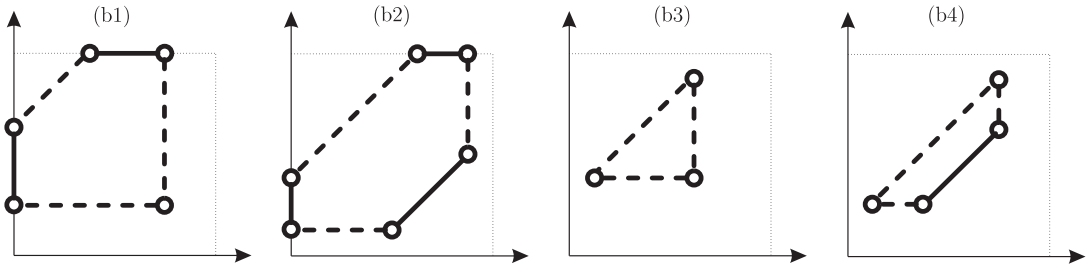}
\caption{Shapes of optimal sets of $\big(\frac{1}p,\frac{1}q\big)$ 
	for which ${I}_H^{\a,\s}$ is $L^p-L^q$ bounded when 
	$\s < 1/2$ and $\a < -1/2$;
	(b1) the case of $\s \ge \a+1$ and $\s > -\a/2-1/4$,
	(b2) the case of $-\a/2-1/4 < \s < \a +1$,
	(b3) the case of $\a+1 \le \s \le -\a/2-1/4$,
	(b4) the case of $\s < \a+1$ and $\s \le -\a/2-1/4$.
	Particular choices of $\a$ and $\s$ are different for each picture.
	}
 \label{fig2}
\end{figure}

\subsection{Dunkl-Laguerre setting} \label{ssec:dunkl}
Let $\a > -1$. The generalized Hermite functions are given by
$$
h_n^{\a}(x) = \frac{1}{\sqrt{2}} 
	\begin{cases}
		(-1)^{n/2} \ell_{n/2}^{\a}(x), & n \; \textrm{even},\\
		(-1)^{(n-1)/2} x\ell_{(n-1)/2}^{\a+1}(x), & n \; \textrm{odd},
	\end{cases}
$$
where $\ell_n^{\a}$ are the Laguerre functions of convolution type naturally extended to $\mathbb{R}$ 
as even functions. The system $\{h_n^{\a}:n \ge 0\}$ is an orthonormal basis in $L^2(dw_{\a})$, where
$w_{\a}$ is the even extension of $\mu_{\a}$,
$$
dw_{\a}(x) = |x|^{2\a+1}\, dx, \qquad x \in \mathbb{R}.
$$
Notice that expansions of even functions with respect to $\{h_n^{\a}\}$ reduce to expansions with respect
to the Laguerre system $\{\ell_n^{\a}\}$.
The $h_n^{\a}$ are eigenfunctions of the one-dimensional 
Dunkl harmonic oscillator 
$$
L_{\a}^{D}f(x) = L_{\a}f(x) + (\a+1/2)\frac{f(x)-f(-x)}{x^2}
$$
(notice that this is a differential-difference `Laplacian')
and one has $L_{\a}^D h_n^{\a} = (2n+2\a+2)h_n^{\a}$. We use the same symbol to denote the natural,
in this situation, self-adjoint extension of $L_{\a}^D$.

For $\a=-1/2$ this setting coincides with that of classical Hermite function expansions. Note that the
parameter $\a$ represents the so-called (in the Dunkl theory) multiplicity function. This function
is trivial when $\a=-1/2$, positive when $\a>-1/2$, and negative for $\a< -1/2$. The latter case is
exotic in the sense that positivity of the multiplicity function is crucial in several important aspects
of the Dunkl theory. As we shall see, the positivity turns out to be meaningful also in our developments.

The Dunkl-Laguerre heat semigroup $\{\exp(-tL_{\a}^D)\}$ possesses an integral representation
(with integration against $dw_{\a}$), and the integral kernel of $\exp(-t L_{\a}^D)$ is
$$
{G}_t^{\a,D}(x,y) = \frac{1}{2}(\sinh 2t)^{-\a-1} \exp\Big( -\frac{1}2 \coth(2t)(x^2+y^2)\Big)
	\Phi_{\a}\Big(\frac{xy}{\sinh 2t}\Big), \qquad x,y \in \mathbb{R},
$$
where $\Phi_{\a}$ is a continuous function on the real line defined by
$$
\Phi_{\a}(u) := |u|^{-\a}\big[ I_{\a}(|u|) + \sgn(u) I_{\a+1}(|u|)\big],
$$
with the value $\Phi_{\a}(0) = 2^{-\a}/\Gamma(\a+1)$ understood in the limiting sense.
By the standard asymptotics (cf.\ \cite[(5.16.4), (5.16.5)]{Leb})
\begin{equation} \label{asym}
I_{\nu}(u) \simeq u^{\a}, \quad u \to 0^+, \qquad \textrm{and} 
\qquad I_{\nu}(u) \simeq u^{-1/2}e^u, \quad u \to \infty,
\end{equation}
and strict positivity of $I_{\nu}(u)$, $u > 0$,
it is straightforward to see that $|\Phi_{\a}(u)| \lesssim |u|^{-\a}I_{\a}(|u|)$, $u\in \mathbb{R}$,
and $\Phi_{\a}(u) \simeq u^{-\a}I_{\a}(u)$, $u > 0$. This together with \eqref{hkl} implies
\begin{align}
|G_t^{\a,D}(x,y)| & \lesssim G_t^{\a}(|x|,|y|), \qquad x,y \in \mathbb{R}, \label{kke1}\\
G_t^{\a,D}(x,y) & \simeq G_t^{\a}(x,y), \qquad x,y > 0. \label{kke2}
\end{align}
A more detailed analysis (see Section \ref{ssec:de} below) reveals that $G_t^{\a,D}(x,y)$ is strictly positive
when either $\a \ge -1/2$ or $xy \ge 0$, but for $\a<-1/2$ and $xy < 0$ it attains also negative
values. It seems that this phenomenon has not been properly noticed before.

As in the previous settings, for $\s>0$ we consider the potential operator
$$
I_{D}^{\a,\s}f(x) = \int_{-\infty}^{\infty} K_{D}^{\a,\s}(x,y)f(y)\, dw_{\a}(y), \qquad x \in \mathbb{R},
$$
where the potential kernel is expressed via $G_t^{\a,D}(x,y)$ as
$$
K_{D}^{\a,\s}(x,y) = \frac{1}{\Gamma(\s)} \int_{0}^{\infty} G_t^{\a,D}(x,y) t^{\s-1}\, dt, 
	\qquad x,y \in \mathbb{R}.
$$
In view of \eqref{kke1} and \eqref{kke2}, we have
\begin{align}
|K_D^{\a,\s}(x,y)| & \lesssim K^{\a,\s}(|x|,|y|), \qquad x,y \in \mathbb{R}, \label{kkp1}\\
K_D^{\a,\s}(x,y) & \simeq K^{\a,\s}(x,y), \qquad x,y > 0. \label{kkp2}
\end{align}
These two relations deliver enough information to obtain a characterization of $L^p-L^q$ boundedness
of $I^{\a,\s}_D$. Nevertheless, the question of an exact description of $K_{D}^{\a,\s}(x,y)$ is an
important problem in its own right.
The result below provides qualitatively sharp estimates of $K_{D}^{\a,\s}(x,y)$ for $\a>-1/2$
(the case of a positive multiplicity function).
\begin{thm} \label{thm:d_ker}
Let $\a>-1/2$ and let $\s > 0$. The following estimates hold uniformly in $x,y \in \mathbb{R}$.
\begin{itemize}
\item[(A)] Assume that $xy \ge 0$, i.e. $x$ and $y$ have the same sign.
\begin{itemize}
\item[(A1)] If $|x|+|y| \le 1$, then
\begin{align*}
{K}_D^{\a,\s}(x,y) & \simeq  \chi_{\{\s>\a+1\}} + \chi_{\{\s=\a+1\}} \log\frac{1}{|x|+|y|} \\
& \quad	 + (|x|+|y|)^{-2\a-1} 
	\begin{cases}
		|x-y|^{2\s-1}, & \s< 1/2,\\
		1+\log \frac{|x|+|y|}{|x-y|}, & \s=1/2,\\
		(|x|+|y|)^{2\s-1}, & \s>1/2.
	\end{cases}
\end{align*}
\item[(A2)] If $|x|+|y| > 1$, then
\begin{align*}
{K}_D^{\a,\s}(x,y) & \simeq \simeq (|x|+|y|)^{-2\a-1} \exp\big(-c|x-y||x+y|\big) \\
& \qquad \times	\begin{cases}
		|x-y|^{2\s-1}, & \s< 1/2,\\
		1+\log^+ \frac{1}{|x-y||x+y|}, & \s=1/2,\\
		(|x|+|y|)^{1-2\s}, & \s>1/2.
	\end{cases}
\end{align*}
\end{itemize}
\item[(B)] Assume that $xy< 0$, i.e. $x$ and $y$ have opposite signs.
\begin{itemize}
\item[(B1)] If $|x|+|y| \le 1$, then
\begin{align*}
{K}_D^{\a,\s}(x,y) & \simeq  \chi_{\{\s>\a+1\}} + \chi_{\{\s=\a+1\}} \log\frac{1}{|x|+|y|} \\
& \quad	 + (|x|+|y|)^{-2\a-1} (|x|+|y|)^{2\s-1}.
\end{align*}
\item[(B2)] If $|x|+|y| > 1$, then
\begin{align*}
{K}_D^{\a,\s}(x,y) & \simeq \simeq (|x|+|y|)^{-2\a-1} \exp\big(-c|x-y||x+y|\big) \\
& \qquad \times		(|x|+|y|)^{1-2\s} (|x|+|y|)^{-4}.
\end{align*}
\end{itemize}
\end{itemize}
\end{thm}
For the sake of completeness, we recall that (see \cite[Theorem 2.4]{NoSt2})
\begin{equation} \label{tmf}
{K}_D^{-1/2,\s}(x,y) \simeq \simeq \exp\big(-c|x-y|(|x|+|y|)\big)
	\begin{cases}
		|x-y|^{2\s-1}, & \s< 1/2,\\
		1+\log^+ \frac{1}{|x-y|(|x|+|y|)}, & \s=1/2,\\
		(1+|x+y|)^{1-2\s}, & \s>1/2,
	\end{cases}
\end{equation}
uniformly in $x,y \in \mathbb{R}$. When $xy \ge 0$, this agrees with the estimates of 
Theorem \ref{thm:d_ker} (A) taken with $\a=-1/2$. On the
other hand, in case $xy < 0$ the exponential factor in \eqref{tmf} possesses essentially better
decay if compared with the exponential factor in Theorem \ref{thm:d_ker} (B2).  
In particular, on the line $y=-x$ the right-hand side in \eqref{tmf} has an exponential decay, 
which is not the case of the right-hand side in (B2). 
This reflects a discontinuity in the behavior of $\Phi_{\a}(u)$ as $\a \to (-1/2)^{+}$
(exponential growth/decay when $u \to -\infty$), see Section \ref{ssec:de} below.

When $\a< -1/2$ (the case of a negative multiplicity function), no general sharp estimates in the spirit
of Theorem \ref{thm:d_ker} are possible, because $K_D^{\a,\s}(x,y)$ attains also negative values.
In fact, we have the following result (the proof can be found in Section \ref{ssec:de}).
\begin{propo} \label{prop:neg}
Let $-1 < \a < -1/2$ and $\s > 0$ be fixed. There exists an unbounded set
$\mathcal{D} \subset \{(x,y): xy < 0\} \subset \mathbb{R}^2$ of positive Lebesgue measure such that
$$
{K}_D^{\a,\s}(x,y) < 0, \qquad (x,y) \in \mathcal{D}.
$$
\end{propo}
Finally, we establish a sharp description of $L^p-L^q$ boundedness of $I_{D}^{\a,\s}$. It occurs that
$I_{D}^{\a,\s}$ behaves exactly in the same way as $I^{\a,\s}$, see Figure \ref{fig1}. The proof
is based on \eqref{kkp1}, \eqref{kkp2} and Theorem \ref{thm:LpLqlag}. Perhaps a bit unexpectedly, 
Theorem \ref{thm:d_ker} is not needed here.
\begin{thm} \label{thm:lag_D}
Let $\a > -1$, $\s>0$ and $1\le p,q \le \infty$.
\begin{itemize}
\item[(a)] If $\a \ge -1/2$, then $I_D^{\a,\s}$ is bounded from $L^p(dw_{\a})$ to $L^q(dw_{\a})$
if and only if
$$
\frac{1}{p} - \frac{\s}{\a+1} \le \frac{1}q < \frac{1}{p} + \frac{\s}{\a+1} \quad \textrm{and} \quad
\bigg(\frac{1}p,\frac1{q}\bigg) \notin 
	\bigg\{ \Big(\frac{\s}{\a+1},0\Big),\Big(1,1-\frac{\s}{\a+1}\Big)\bigg\}.
$$
\item[(b)] If $\a < -1/2$, then $I_D^{\a,\s}$ is bounded from $L^p(dw_{\a})$ to $L^q(dw_{\a})$
if and only if
$$
\frac{1}p + \frac{\s}{\a} \le \frac{1}q < \frac{1}p + \frac{\s}{\a+1}.
$$
\end{itemize}
\end{thm}
Note that the result in the Hermite case $\a=-1/2$ was known earlier, see \cite{NoSt2} and references
given there. Essentially, also the sufficiency part of Theorem \ref{thm:lag_D} (a) was known before, even in
the multi-dimensional setting, see \cite[Theorem 6.1]{NoSt1}. The rest of the theorem is new.

\begin{rem}
The estimates of Theorem \ref{thm:conv_ker} specified to $\a=\pm 1/2$ and the estimates \eqref{tmf} for the
harmonic oscillator potential kernel are consistent in the following way.
Theorem \ref{thm:conv_ker} implies \eqref{tmf} for $xy>0$, as can be verified by means of the well-known
relation
$$
2 K_{D}^{-1/2,\s}(x,y) = K^{-1/2,\s}(x,y)+xyK^{1/2,\s}(x,y), \qquad x,y > 0.
$$
On the other hand, \eqref{tmf} implies the bounds of Theorem \ref{thm:conv_ker} specified to $\a=-1/2$,
since
$$
K^{-1/2,\s}(x,y) = K_D^{-1/2,\s}(x,y) + K_D^{-1/2,\s}(-x,y), \qquad x,y > 0.
$$
\end{rem}

\section{Estimates of the potential kernels} \label{sec:pot}

In this section we prove Theorems \ref{thm:conv_ker} and \ref{thm:d_ker}, and also Proposition \ref{prop:neg}.
We begin with two auxiliary technical results that provide sharp description of the behavior of the integrals
$J_A(T,S)$ and $E_A(T,S)$ defined below. These are essentially \cite[Lemmas 2.1-2.3]{NoSt2}. 
Here we give slightly more general statements, 
nevertheless their proofs are almost the same as those in \cite{NoSt2}.

Let
\begin{align*}
J_A(T,S) & := \int_T^S t^A \exp(-t)\, dt, \qquad 0 \le T \le S \le\infty, \quad T  < \infty,\\
E_A(T,S) & := \int_0^1 t^A \exp\big(-Tt^{-1}-St\big)\, dt, \qquad 0\le T, S <\infty. 
\end{align*}

\begin{lem}[{\cite[Lemma 2.1, Lemma 2.2]{NoSt2}}] \label{lem:J}
Let $A \in \mathbb{R}$, $\beta>1$ and $\gamma > 0$ be fixed.
The following estimates hold uniformly in $0\le T\le S \le \infty$, $T< \infty$.
\begin{itemize}
\item[(a)] If $S \le \beta T$, then
$$
J_A(T,S) \simeq\simeq T^A (S-T)\exp(-cT).
$$
\item[(b)] If $S>\beta T$ and $T \ge \gamma$, then
$$
J_A(T,S) \simeq T^A \exp(-T).
$$
\item[(c)] If $S>\beta T$ and $S \ge \beta \gamma$ and $T<\gamma$, then
$$
J_A(T,S) \simeq \begin{cases}
									T^{A+1}, & A< -1,\\
									1+\log^{+}(1/T), & A=-1,\\
									1, & A>-1.
								\end{cases}
$$
\item[(d)] If $S>\beta T$ and $S<\beta \gamma$, then
$$
J_A(T,S) \simeq \begin{cases}
									T^{A+1}, & A < -1,\\
									\log(S/T), & A=-1,\\
									S^{A+1}, & A> -1.
								\end{cases}
$$
\end{itemize}
\end{lem}

\begin{lem}[{\cite[Lemma 2.3]{NoSt2}}] \label{lem:E}
Let $A \in \mathbb{R}$ and $\gamma > 0$ be fixed. Then
$$
E_A(T,S) \simeq \simeq \exp\Big( -c\sqrt{T(T\vee S)}\Big)
\begin{cases}
T^{A+1}, & A< -1,\\
1+\log^+ \frac{1}{T(T\vee S)}, & A = -1,\\
(S\vee \gamma)^{-A-1}, & A>-1,
\end{cases}
$$
uniformly in $T,S \ge 0$.
\end{lem}

\subsection{Estimates of the Laguerre potential kernels}
Proving Theorem \ref{thm:conv_ker} requires some further preparation. 
The plan is to estimate $K^{\a,\s}(x,y)$ first in terms of $J_A(T,S)$ and $E_{A}(T,S)$,
and then to apply Lemmas \ref{lem:J} and \ref{lem:E}.
Let
\begin{align*}
G_t(x,y) & = 
	\frac{1}{(2\pi \sinh 2t)^{1/2}} \exp\bigg( -\frac{1}4 \big[ \tanh(t)\, 
		(x+y)^2 + \coth(t)\,(x-y)^2\big]\bigg) \\
& = 	\frac{1}{(2\pi \sinh 2t)^{1/2}} 
	\exp\bigg( -\frac{1}2 \coth(2t) \,\big(x^2+y^2\big) + \frac{xy}{\sinh 2t}\bigg)
\end{align*}
be the heat kernel associated with one-dimensional Hermite function expansions;
note that $G_t(x,y)=G_t^{-1/2,D}(x,y)$.
The behavior of $G_t^{\a}(x,y)$ can be described in a sharp way in terms of $G_t(x,y)$.
\begin{lem} \label{lem:her_lag}
Let $\a>-1$. Then
$$
G_t^{\a}(x,y) \simeq (xy \vee \sinh 2t)^{-\a-1/2} G_t(x,y)
$$
uniformly in $x,y>0$ and $t>0$.
\end{lem}

\begin{proof}
Elementary exercise based on the asymptotics \eqref{asym}.
\end{proof}

\begin{lem} \label{lem:exp_JE}
Let $\a>-1$. The following estimates hold uniformly in $x,y>0$.
\begin{itemize}
\item[(a)] If $xy \le 1$, then
\begin{align*}
K^{\a,\s}(x,y) \simeq \simeq &\, \exp\big(-c(x+y)^2\big) + (x+y)^{2\s-2\a-2} J_{\a-\s}\bigg(
	c_1(x+y)^2, c_2 \frac{(x+y)^2}{xy} \bigg) \\
	&\,+ (xy)^{\s-\a-1} E_{\s-3/2}\bigg( c\frac{(x-y)^2}{xy}, c xy(x+y)^2 \bigg),
\end{align*}
where $c_1 < c_2$ are positive constants, independent of $x$ and $y$, that may be different in the lower
and upper estimate.
\item[(b)] If $xy>1$, then
$$
K^{\a,\s}(x,y) \simeq \simeq \exp\big(-c(x+y)^2\big) + (xy)^{-\a-1/2} E_{\s-3/2}\big(
	c(x-y)^2,c(x+y)^2\big).
$$
\end{itemize}
\end{lem}

\begin{proof}
In view of Lemma \ref{lem:her_lag}, 
$$
K^{\a,\s}(x,y) \simeq \int_0^{\infty} (xy \vee \sinh 2t)^{-\a-1/2} G_t(x,y) t^{\s-1}\, dt.
$$
To proceed, we get rid of the maximum above by splitting the integral according to the point
$p(xy)$, where the function $p(r)$ is defined by the identity $\sinh 2p(r) = r$, $r >0$. 
Notice that
$$
p(xy) \simeq \begin{cases}
								xy, & xy \le 1,\\
								\log 2xy, & xy > 1.
							\end{cases}
$$
Taking into account the explicit form of $G_t(x,y)$, we can write
\begin{align*}
& K^{\a,\s}(x,y) \\
&  \simeq   (xy)^{-\a-1/2}\int_0^{p(xy)} (\sinh 2t)^{-1/2} t^{\s-1}
	\exp\Big( -\frac{1}{4} \big[ \tanh(t)\, (x+y)^2 + \coth(t)\,(x-y)^2\big] \Big) \, dt\\
	& \qquad + \int_{p(xy)}^{\infty}(\sinh 2t)^{-\a-1} t^{\s-1} \exp\Big( -\frac{1}2 \coth(2t)\, 
		\big(x^2+y^2\big) \Big) \, dt \\
	& \equiv \mathcal{I}_0 + \mathcal{I}_{\infty}.
\end{align*}

We first prove (a). To this end assume that $xy \le 1$. Since $p(xy) \lesssim xy$, we have
\begin{equation} \label{bt}
\mathcal{I}_0 \simeq\simeq (xy)^{-\a-1/2} \int_0^{p(xy)} t^{\s-3/2} \exp\Big( -c \big[ t^{-1}(x-y)^2
	+ t(x+y)^2 \big] \Big) \, dt.
\end{equation}
Changing the variable of integration, we get
$$
\mathcal{I}_0 \simeq \simeq (xy)^{-\a-1/2} p(xy)^{\s-1/2} \int_0^1 s^{\s-3/2}
	\exp\bigg( -c\bigg[ s^{-1} \frac{(x-y)^2}{p(xy)} + s p(xy) (x+y)^2 \bigg]\bigg)\, ds.
$$
Since $p(xy)\simeq xy$, we conclude that 
$$
\mathcal{I}_0 \simeq\simeq (xy)^{\s-\a-1} E_{\s-3/2}\bigg( c\frac{(x-y)^2}{xy},cxy(x+y)^2\bigg).
$$
Next, we estimate $\mathcal{I}_{\infty}$. We split this integral according to $\mathcal{C}$ satisfying
$2 p(xy) \le \mathcal{C}xy$, $xy \le 1$, obtaining
\begin{align*}
\mathcal{I}_{\infty} & \simeq \simeq \int_{p(xy)}^{\mathcal{C}} t^{\s-\a-2} \exp\Big(-c\frac{1}t
	\big(x^2+y^2\big)\Big)\, dt 
 + \exp\big(-c(x^2+y^2)\big) \int_{\mathcal{C}}^{\infty} t^{\s-1} e^{-2(\a+1)t}\, dt \\
& \equiv \mathcal{I}_{\infty,1} + \mathcal{I}_{\infty,2}.
\end{align*}
Treatment of $\mathcal{I}_{\infty,2}$ is obvious since the integral over $(\mathcal{C},\infty)$ is a finite
constant. To deal with $\mathcal{I}_{\infty,1}$ we change the variable of integration and find that
$$
\mathcal{I}_{\infty,1} \simeq \big(x^2+y^2\big)^{\s-\a-1} 
	\int_{c\frac{x^2+y^2}{\mathcal{C}}}^{c\frac{x^2+y^2}{p(xy)}} s^{\a-\s}e^{-s}\, ds.
$$
Since $p(xy) \simeq xy$ and $x^2+y^2 \simeq (x+y)^2$, we infer that
$$
\mathcal{I}_{\infty} \simeq\simeq (x+y)^{2\s-2\a-2} J_{\a-\s}\bigg( c_1(x+y)^2,c_2 \frac{(x+y)^2}{xy}\bigg)
	+ \exp\big(-c(x+y)^2\big),
$$
where $c_1 < c_2$ are positive constants that may differ in the lower and upper estimate.
Item (a) follows.

To prove (b), assume that $xy>1$. Consider first $\mathcal{I}_{\infty}$. Since $p(xy) \gtrsim \log 2xy$, we
can write
$$
\mathcal{I}_{\infty} \simeq\simeq \exp\big( -c(x^2+y^2)\big) \int_{p(xy)}^{\infty} t^{\s-1}e^{-2(\a+1)t}\,dt.
$$
But $t^{\s-1}e^{-2(\a+1)t} \simeq \simeq e^{-ct}$ for $t > p(1)$, so,
taking into account that actually $p(xy) \simeq \log 2xy$, we see that
$$
(xy)^{-c_3} \lesssim \int_{p(xy)}^{\infty} t^{\s-1}e^{-2(\a+1)t}\,dt \lesssim (xy)^{-c_4},
$$
where $c_3$ and $c_4$ are positive constants. Thus
$$
\mathcal{I}_{\infty} \simeq\simeq \exp\big(-c(x+y)^2\big).
$$
Finally, we analyze $\mathcal{I}_0$. We split this integral getting 
\begin{align*}
\mathcal{I}_0 & \simeq\simeq (xy)^{-\a-1/2} \int_0^{p(1)} t^{\s-3/2} \exp\Big( -c\big[ t^{-1}(x-y)^2
	+ t(x+y)^2\big]\Big) \, dt \\
& \qquad + (xy)^{-\a-1/2} \exp\big(-c(x^2+y^2)\big) \int_{p(1)}^{p(xy)} t^{\s-1} e^{-t}\, dt \\
& \equiv \mathcal{I}_{0,1} + \mathcal{I}_{0,2}.
\end{align*}
Clearly,
$$
\mathcal{I}_{0,2} \lesssim (xy)^{-\a-1/2} \exp\big(-c(x^2+y^2)\big) 
	\int_{p(1)}^{\infty} t^{\s-1} e^{-t}\, dt \simeq\simeq \exp\big(-c(x+y)^2\big).
$$
Further, changing the variable of integration we see that
$$
\mathcal{I}_{0,1} \simeq\simeq (xy)^{-\a-1/2} E_{\s-3/2}\big( c(x-y)^2,c(x+y)^2\big).
$$
Altogether, the above estimates justify item (b).
\end{proof}

We are now in a position to prove Theorem \ref{thm:conv_ker}.
\begin{proof}[{Proof of Theorem \ref{thm:conv_ker}}]
We distinguish three cases.\\ 
\textbf{Case 1: $\boldsymbol{xy > 1}$.} Notice that in this case $x+y >1$, so we must show that
\begin{equation} \label{ee}
K^{\a,\s}(x,y) \simeq \simeq (x+y)^{-2\a-1} \exp\big(-c|x-y|(x+y)\big) 
	\begin{cases}
		|x-y|^{2\s-1}, & \s< 1/2,\\
		1+\log^+ \frac{1}{|x-y|(x+y)}, & \s=1/2,\\
		(x+y)^{1-2\s}, & \s>1/2.
	\end{cases}
\end{equation}
By Lemma \ref{lem:exp_JE} we know that
$$
K^{\a,\s}(x,y) \simeq \simeq \exp\big(-c(x+y)^2\big) + (xy)^{-\a-1/2} E_{\s-3/2}\big(
	c(x-y)^2,c(x+y)^2\big).
$$
Here $E_{\s-3/2}$ can be estimated by means of Lemma \ref{lem:E}, we get
\begin{align*}
K^{\a,\s}(x,y) & \simeq \simeq \exp\big(-c(x+y)^2\big) \\
	& \qquad +(xy)^{-\a-1/2} \exp\big(-c|x-y|(x+y)\big) 
	\begin{cases}
		|x-y|^{2\s-1}, & \s< 1/2,\\
		1+\log^+ \frac{1}{|x-y|(x+y)}, & \s=1/2,\\
		(x+y)^{1-2\s}, & \s>1/2.
	\end{cases}
\end{align*}
In this sum the first term can be neglected since, as easily verified, it contributes to the relation
$\simeq\simeq$ no more than the second one. Further, the factor $(xy)^{-\alpha-1/2}$ can be replaced
by $(x+y)^{-2\alpha-1}$. This is clear when $x$ and $y$ are comparable. In the opposite case, 
it suffices to take into account the bounds $(x+y)^2 \simeq |x-y|(x+y) \gtrsim xy$, recall that $xy> 1$
and use the exponential decay. The conclusion follows.\\
\textbf{Case 2: $\boldsymbol{xy \le 1}$ and $\boldsymbol{x+y > 3}$.}
Again, our aim is to prove \eqref{ee}. Observe that in this case $x$ and $y$ are non-comparable.
For symmetry reasons, we may assume that $x > 2y$. Thus the estimate to be shown is
$$
K^{\a,\s}(x,y) \simeq\simeq x^{-2\a-1} \exp\big(-c x^2\big)
	\begin{cases}
		x^{2\s-1}, & \s < 1/2,\\
		1+\log^{+}\frac{1}x, & \s=1/2,\\
		x^{1-2\s}, & \s > 1/2.
	\end{cases}
$$
But $x>1$, so it is enough to check that
\begin{equation} \label{tc}
K^{\a,\s}(x,y) \simeq\simeq \exp\big(-c x^2\big).
\end{equation}

By Lemma \ref{lem:exp_JE},
\begin{align*}
K^{\a,\s}(x,y) & \simeq \simeq  \exp\big(-c(x+y)^2\big) + (x+y)^{2\s-2\a-2} J_{\a-\s}\bigg(
	c_1(x+y)^2, c_2 \frac{(x+y)^2}{xy} \bigg) \\
	&\qquad + (xy)^{\s-\a-1} E_{\s-3/2}\bigg( c\frac{(x-y)^2}{xy}, c xy(x+y)^2 \bigg) \\
& \equiv U_1(x,y) + U_2(x,y) + U_3(x,y).
\end{align*}
Here $U_1$ agrees with the right-hand side of \eqref{tc}, 
so it suffices to bound suitably $U_2$ and $U_3$ from above.
Observe that $U_2$ can be estimated from
above by replacing the second argument of $J_{\a-\s}$ by $\infty$. Then, taking into account that
$c_1 (x+y)^2 \gtrsim x^2 > 1$, Lemma \ref{lem:J} (b) shows that 
$$
U_2(x,y)  \lesssim
x^{2\s-2\a-2} x^{2\a-2\s} \exp\big(-c_1 x^2\big) < \exp\big(-c_1 x^2\big).
$$
Thus this term also fits to \eqref{tc} contributing in the sense of $\simeq\simeq$ no more than the first one.
Finally, $U_3$ can be estimated from above by replacing the second argument of $E_{\s-3/2}$ by $0$.
Then, with the aid of Lemma \ref{lem:E} and the relations $|x-y|\simeq x+y \simeq x > 1$ we get
\begin{align*}
U_3(x,y) & \lesssim
(xy)^{\s-\a-1} \exp\bigg( - \tilde{c} \frac{x^2}{xy}\bigg)
	{\begin{cases}
		\left(\frac{x^2}{xy}\right)^{\s-1/2}, & \s < 1/2,\\
		1, & \s\ge 1/2
	\end{cases}}
	\\
	& \lesssim \exp\big( -\tilde{c} x^2/2\big),
\end{align*}
the last estimate being a consequence of the inequalities $xy \le 1$ and $x>1$.
Now \eqref{tc} follows.\\
\textbf{Case 3: $\boldsymbol{xy \le 1}$ and $\boldsymbol{x+y \le 3}$.} In this case $x \le 3$ and $y \le 3$.
Since the estimates of (i) and (ii) of Theorem \ref{thm:conv_ker} essentially 
coincide for $x$ and $y$ separated from $0$ and $\infty$, what we need to prove is 
\begin{equation} \label{ex}
K^{\a,\s}(x,y) \simeq  \chi_{\{\s>\a+1\}} + \chi_{\{\s=\a+1\}} \log^+\frac{1}{x+y} 
	 + (x+y)^{-2\a-1} 
	\begin{cases}
		|x-y|^{2\s-1}, & \s< 1/2,\\
		1+\log \frac{x+y}{|x-y|}, & \s=1/2,\\
		(x+y)^{2\s-1}, & \s>1/2.
	\end{cases}
\end{equation}

We keep using the description of $K^{\a,\s}$ in terms of $U_1$, $U_2$ and $U_3$, see above.
Observe that 
$$
U_1(x,y) \simeq 1.
$$
To estimate $U_2$ we apply Lemma \ref{lem:J} (b)--(d). 
After some elementary manipulations, taking into account that
$x$ and $y$ stay bounded, we get
\begin{equation} \label{u2}
U_2(x,y) \simeq \begin{cases}
									(x+y)^{2\s-2\a-2}, & \s < \a +1,\\
									1+\log^{+}\frac{1}{x+y}, & \s=\a+1,\\
									1, & \s > \a + 1.
								\end{cases}
\end{equation}
Considering $U_3$, by the boundedness of $x$ and $y$ and the structure of the integral $E_{\s-3/2}$ we
may assume that its second argument is $0$. Then, in view of Lemma \ref{lem:E}, we have
$$
U_3(x,y) \simeq \simeq (xy)^{-\a-1/2} \exp\bigg( -c \frac{(x-y)^2}{xy}\bigg)
	\begin{cases}
		|x-y|^{2\s-1}, & \s < 1/2,\\
		1+\log^{+} \frac{xy}{(x-y)^2}, & \s=1/2,\\
		(xy)^{\s-1/2}, & \s>1/2.
	\end{cases}
$$
We claim that, excluding the exponential factor, all the products $xy$ here can be replaced by $(x+y)^2$.
Indeed, this is clear when $x$ and $y$ are comparable. In the opposite case, say when $x> 2y$,
$\log^+$ is controlled by a constant, so its argument can be replaced by $(x+y)^2/(x-y)^2\simeq 1$.
Further, we have $(x-y)^2/{xy} \simeq {x}/y$, and given any $\gamma\in\mathbb{R}$ and 
$\mathcal{C}>0$ fixed
$$
(xy)^{\gamma} \exp\Big(-\mathcal{C}\frac{x}y\Big) = 
	x^{2\gamma} \Big(\frac{x}{y}\Big)^{-\gamma} \exp\Big(-\mathcal{C}\frac{x}y\Big) \simeq\simeq
		(x+y)^{2\gamma} \exp\Big(-c\frac{x}y\Big).
$$
The claim follows and we conclude that
\begin{equation} \label{u3}
U_3(x,y) \simeq \simeq (x+y)^{-2\a-1} \exp\bigg( -c \frac{(x-y)^2}{xy}\bigg)
	\begin{cases}
		|x-y|^{2\s-1}, & \s < 1/2,\\
		1+\log \frac{x+y}{|x-y|}, & \s=1/2,\\
		(x+y)^{2\s-1}, & \s>1/2.
	\end{cases}
\end{equation}

Assume that $x \simeq y$. Then the exponential factor on the right-hand side above is roughly a constant. Moreover,
$U_3(x,y) \gtrsim (x+y)^{2\s-2\a-2}$. Therefore, in view of \eqref{u2} and \eqref{u3}, 
for comparable $x$ and $y$
\begin{align*}
U_2(x,y) + U_3(x,y) & \simeq \chi_{\{\s>\a+1\}} + \chi_{\{\s=\a+1\}} \bigg(1+\log^+\frac{1}{x+y}\bigg) \\
	& \quad + (x+y)^{-2\a-1}
	\begin{cases}
		|x-y|^{2\s-1}, & \s < 1/2,\\
		1+\log \frac{x+y}{|x-y|}, & \s=1/2,\\
		(x+y)^{2\s-1}, & \s>1/2.
	\end{cases}
\end{align*}
Notice that the right-hand side here is separated from $0$, and this remains true even without the second term. 
Thus $U_1(x,y)\lesssim U_2(x,y)+U_3(x,y)$ and the second term 
can be replaced by $\chi_{\{\s=\a+1\}}\log^+\frac{1}{x+y}$. We see that \eqref{ex} holds when $x \simeq y$.

Finally, let $x$ and $y$ be non-comparable. For symmetry reasons, we may assume that $x > 2y$.
Then the desired estimate \eqref{ex} takes the form
\begin{equation} \label{exn}
K^{\a,\s}(x,y) \simeq \chi_{\{\s> \a+1\}} + \chi_{\{\s=\a+1\}}\log^+\frac{1}{x} + x^{2\s-2\a-2}.
\end{equation}
On the other hand, from \eqref{u2} and \eqref{u3} we have
\begin{align*}
U_2(x,y) + U_3(x,y) & \simeq\simeq \chi_{\{\s>\a+1\}} + \chi_{\{\s=\a+1\}}\bigg(1+\log^+\frac{1}{x}\bigg)
	+ \chi_{\{\s<\a+1\}} x^{2\s-2\a-2}\\ & \qquad  + x^{2\s-2\a-2}\exp\bigg(-c\frac{x}y\bigg).
\end{align*}
Observe that the fourth term on the right-hand side here is controlled by the other terms, so it may
be neglected. Moreover, the sum of the first three terms is separated from $0$ and thus controls $U_1(x,y)$.
This means that
$$
K^{\a,\s}(x,y) \simeq \chi_{\{\s>\a+1\}} + \chi_{\{\s=\a+1\}}\bigg(1+\log^+\frac{1}{x}\bigg)
	+ \chi_{\{\s<\a+1\}} x^{2\s-2\a-2}.
$$
Here we can neglect $\chi_{\{\s<\a+1\}}$ since $x^{2\s-2\a-2} \lesssim 1$ for $\s \ge \a+1$.
After that one can also replace $1+\log^+(1/x)$ by $\log^+(1/x)$ since $x^{2\s-2\a-2} \equiv 1$
when $\s=\a+1$. Thus we arrive at \eqref{exn}.

The proof of Theorem \ref{thm:conv_ker} is complete.
\end{proof}

\subsection{Estimates of the Dunkl potential kernel} \label{ssec:de}
We first focus our attention on the Dunkl heat kernel $G_t^{\a,D}(x,y)$.
Recall that this kernel is defined by means of the auxiliary function
$$
\Phi_{\a}(u) = 
	|u|^{-\a}\big[ I_{\a}(|u|) + \sgn(u) I_{\a+1}(|u|)\big].
$$
As we saw in Section \ref{ssec:dunkl}, $\Phi_{\a}(u) \simeq u^{-\a}I_{\a}(u)$, $u \ge 0$,
with the value at $u=0$ understood in a limiting sense. However, for $u < 0$ the situation is more subtle,
because of the cancellation occurring in the difference of the Bessel functions. 
Thus we now analyze the function
$$
\Psi_{\a}(u):= I_{\a}(u) - I_{\a+1}(u), \qquad u >0.
$$
For $\a=-1/2$ this has an explicit form (cf.\ \cite[(5.8.5)]{Leb}) and we have
$$
\Psi_{-1/2}(u) = \sqrt{\frac{2}{\pi u}} \exp(-u);
$$
notice the exponential decay. Further, when $\a<-1/2$, it is not difficult to see that $\Psi_{\a}(u)$ 
is negative for sufficiently large $u$. Indeed, by the standard large argument asymptotics for the
Bessel function (cf.\ \cite[(5.11.10)]{Leb}) we have
$I_{\nu}(u) = (2\pi u)^{-1/2}\exp(u)[1-(\nu-1/2)(\nu+1/2)/(2u)+\mathcal{O}(u^{-2})]$ for large $u$,
hence when $\a \neq -1/2$
\begin{equation} \label{aspsi}
\Psi_{\a}(u) \sim \frac{\a+1/2}{u} I_{\a}(u), \qquad u \to \infty.
\end{equation}
Finally, in case $\a> -1/2$ we use \cite[Theorem 2]{Na} (specified to $L_{\nu,1,0}$ and $U_{\nu,2,0}$;
see \cite[p.\,10]{Na}) getting
$$
\frac{\a+1/2}{\a+1/2+u} < 1 - \frac{I_{\a+1}(u)}{I_{\a}(u)} < \frac{2(\a+1)}{2(\a+1)+u}, \qquad u >0.
$$
This implies
$$
\Psi_{\a}(u) \simeq I_{\a}(u) \big(1\wedge u^{-1}\big), \qquad u >0.
$$
Here, in contrast with the Hermite case $\a=-1/2$, we have an exponential growth as $u \to \infty$.

From the above considerations we draw the following conclusions.
The behavior of $G_t^{\a,D}(x,y)$ is qualitatively different
in the singular case $\alpha=-1/2$ (trivial multiplicity function). 
The case $\alpha < -1/2$ (negative multiplicity functions) is exotic in the sense that the heat
kernel takes also negative values. Indeed, taking into account \eqref{aspsi}, 
we have ${G}_t^{\a,D}(x,y)<0$ when $xy<0$ and $|xy|/\sinh{2t}$ is large enough.
On the other hand, the case $\a>-1/2$ is more standard.
With the aid of Lemma~\ref{lem:her_lag} and \eqref{hkl}
we can describe the behavior of ${G}_t^{\a,D}(x,y)$ in terms of the Hermite heat kernel 
$G_t(x,y) = G_t^{-1/2,D}(x,y)$.
\begin{propo} \label{thm:her_dun}
Let $\a>-1/2$. The following estimates hold uniformly in $x,y \in \mathbb{R}$ and $t>0$.
\begin{itemize}
\item[(a)] If $xy \ge 0$, then
$$
{G}_t^{\a,D}(x,y) \simeq G_t(|x|,|y|) 
	\begin{cases}
		(\sinh 2t)^{-\a-1/2}, & xy \le \sinh 2t,\\
		(xy)^{-\a-1/2}, & xy > \sinh 2t.
	\end{cases}
$$
\item[(b)] If $xy<0$, then
$$
{G}_t^{\a,D}(x,y) \simeq G_t(|x|,|y|) 
	\begin{cases}
		(\sinh 2t)^{-\a-1/2}, & |xy| \le \sinh 2t,\\
		\sinh(2t)|xy|^{-\a-3/2}, & |xy| > \sinh 2t.
	\end{cases}
$$
\end{itemize}
\end{propo}

Note that item (a) will not be needed for the proof of Theorem \ref{thm:d_ker}, but we state it for
the sake of completeness. On the other hand, (b) is essential, together with 
good estimates of the resulting auxiliary kernel
$$
\widetilde{K}^{\a,\s}(x,y) = \int_0^{\infty} \Big( \frac{\sinh 2t}{xy} \wedge 1\Big) 
	(xy \vee \sinh 2t)^{-\a-1/2} G_t(x,y) t^{\s-1}\, dt, \qquad x,y >0.
$$
\begin{lem} \label{lem:d_exp_JE}
Let $\a>-1$. The following estimates hold uniformly in $x,y>0$.
\begin{itemize}
\item[(a)] If $xy \le 1$, then
\begin{align*}
\widetilde{K}^{\a,\s}(x,y) \simeq \simeq &\, \exp\big(-c(x+y)^2\big) + (x+y)^{2\s-2\a-2} J_{\a-\s}\bigg(
	c_1(x+y)^2, c_2 \frac{(x+y)^2}{xy} \bigg) \\
	&\,+ (xy)^{\s-\a-1} E_{\s-1/2}\bigg( c\frac{(x-y)^2}{xy}, c xy(x+y)^2 \bigg),
\end{align*}
where $c_1 < c_2$ are positive constants, independent of $x$ and $y$, that may be different in the lower
and upper estimate.
\item[(b)] If $xy>1$, then
$$
\widetilde{K}^{\a,\s}(x,y) \simeq \simeq \exp\big(-c(x+y)^2\big) + (xy)^{-\a-3/2} E_{\s-1/2}\big(
	c(x-y)^2,c(x+y)^2\big).
$$
\end{itemize}
\end{lem}

\begin{proof}
Using the notation of the proof of Lemma \ref{lem:exp_JE} and recalling the explicit formulas for $G_t(x,y)$,
we can write
\begin{align*}
& \widetilde{K}^{\a,\s}(x,y) \\
&  \simeq   (xy)^{-\a-3/2}\int_0^{p(xy)} (\sinh 2t)^{1/2} t^{\s-1}
	\exp\Big( -\frac{1}{4} \big[ \tanh(t)\, (x+y)^2 + \coth(t)\,(x-y)^2\big] \Big) \, dt\\
	& \qquad + \int_{p(xy)}^{\infty}(\sinh 2t)^{-\a-1} t^{\s-1} \exp\Big( -\frac{1}2 \coth(2t)\, 
		\big(x^2+y^2\big) \Big) \, dt \\
	& \equiv \mathcal{I}_0 + \mathcal{I}_{\infty}.
\end{align*}
Here $\mathcal{I}_{\infty}$ is the same as in the proof of Lemma \ref{lem:exp_JE}, so we need to analyze
only $\mathcal{I}_0$. 

In case (a) we have $p(xy) \lesssim xy \le 1$, so
$$
\mathcal{I}_0 \simeq\simeq (xy)^{-\a-3/2} \int_0^{p(xy)} t^{\s-1/2} \exp\Big( -c \big[ t^{-1}(x-y)^2
	+ t(x+y)^2 \big] \Big) \, dt.
$$
The right-hand side here coincides with the right-hand side in \eqref{bt} 
after replacing $\a$ by $\a+1$ and $\s$ by $\s+1$. Thus we already know that
$$
\mathcal{I}_0 \simeq\simeq (xy)^{\s-\a-1} E_{\s-1/2}\bigg( c\frac{(x-y)^2}{xy},cxy(x+y)^2\bigg).
$$

Considering (b), when $xy>1$ we have
\begin{align*}
\mathcal{I}_0 & \simeq\simeq (xy)^{-\a-3/2} \int_0^{p(1)} t^{\s-1/2} \exp\Big( -c\big[ t^{-1}(x-y)^2
	+ t(x+y)^2\big]\Big) \, dt \\
& \qquad + (xy)^{-\a-3/2} \exp\big(-c(x^2+y^2)\big) \int_{p(1)}^{p(xy)} t^{\s-1} e^{t}\, dt \\
& \equiv \mathcal{I}_{0,1} + \mathcal{I}_{0,2}.
\end{align*}
As in the proof of Lemma \ref{lem:exp_JE},
$$
\mathcal{I}_{0,1} \simeq\simeq (xy)^{-\a-3/2} E_{\s-1/2}\big( c(x-y)^2,c(x+y)^2\big).
$$
Moreover, since $e^{2p(xy)} \simeq xy$,
\begin{align*}
\mathcal{I}_{0,2} & \lesssim (xy)^{-\a-3/2} \exp\big(-c(x^2+y^2)\big) 
	\int_{p(1)}^{p(xy)} e^{2t}\, dt \\
& \lesssim (xy)^{-\a-1/2} \exp\big(-c(x^2+y^2)\big)
	\simeq\simeq \exp\big(-c(x+y)^2\big).
\end{align*}
The conclusion follows.
\end{proof}

\begin{proof}[{Proof of Theorem \ref{thm:d_ker}}]
Let us first assume  $xy \ge 0$.
Observe that $K^{\a,\s}_D(x,y)=K^{\a,\s}_D(-x,-y)$, so it is enough to consider the case $x,y \ge 0$.
If $x,y>0$, then we easily get the desired estimates by means of \eqref{kkp2} 
and Theorem \ref{thm:conv_ker}.
If $x=0$ or $y=0$, then \eqref{kkp2} still holds, with a limiting understanding of the values of
$K^{\a,\s}(x,y)$ and, implicitly, $G_t^{\a}(x,y)$. Tracing the proof of Theorem \ref{thm:conv_ker},
one can ensure that the asserted bounds for $K^{\a,\s}(x,y)$ remain true for all $x,y \ge 0$, hence the
conclusion again follows.

Assume next that $xy<0$. Taking into account Proposition \ref{thm:her_dun} (b), we infer that
$K_{D}^{\a,\s}(x,y) \simeq \widetilde{K}^{\a,\s}(|x|,|y|)$. On the other hand, the estimates of
Lemma \ref{lem:d_exp_JE} coincide with the estimates of Lemma \ref{lem:exp_JE} with
$\a$ replaced by $\a+1$ and $\s$ replaced by $\s+1$. Thus the behavior of $\widetilde{K}^{\a,\s}(x,y)$
is the same as the behavior of $K^{\a+1,\s+1}(x,y)$ in the sense of the bounds from 
Theorem \ref{thm:conv_ker}. 
Now the conclusion follows by observing that $|x|+|y| = |x-y|$ and $||x|-|y|| = |x+y|$ when $xy < 0$.
\end{proof}

Finally, we prove Proposition \ref{prop:neg}.
\begin{proof}[{Proof of Proposition \ref{prop:neg}}]
Let $x,y\in\mathbb{R}^2$ be such that $xy < 0$. By the asymptotics \eqref{aspsi} it follows that
\begin{equation} \label{t4}
{G}_t^{\a,D}(x,y) \lesssim - G_t(|x|,|y|) \sinh(2t) |xy|^{-\a-3/2}
\end{equation}
provided that $|xy|/\sinh{2t}$ is sufficiently large. We then focus on $x$ and $y$ such that
\eqref{t4} holds uniformly in $t \le p(1)$, and we may assume that $|xy|> 1$. As in the proof of
Lemma \ref{lem:d_exp_JE} (b), we infer that
$$
{K}_D^{\a,\s}(x,y) \le c_1\exp\big( -c_2(|x|+|y|)^2\big) - c_3 |xy|^{-\a-3/2}
	E_{\s-1/2}\big( c_4(|x|-|y|)^2,c_4(|x|+|y|)^2\big)
$$  
for $x,y$ in question, with some positive constants $c_i$, $i=1,\ldots,4$. 
The right-hand side here is certainly
negative when $y=-x$ and $|x|$ is large enough, say $|x|\ge \mathcal{C}>0$, as can be seen from
Lemma~\ref{lem:E}. For continuity reasons, the same must be true for $(x,y)$ laying in a neighborhood of 
the set $\{(x,y): x=-y,\; |x|\ge \mathcal{C}\}$.
\end{proof}

\section{$L^p-L^q$ estimates} \label{sec:LpLq}

This section is devoted to the proofs of Theorems \ref{thm:LpLqlag}, \ref{thm:lag_H} and \ref{thm:lag_D}.
Given $1\le p \le \infty$, we denote by $p'$ its conjugate exponent, $1/p+1/p'=1$.
\subsection{$\boldsymbol{L^p-L^q}$ estimates in the Laguerre setting of convolution type}
Theorem \ref{thm:LpLqlag} follows immediately from the two lemmas 
below that describe sharply $L^p-L^q$ behavior of two auxiliary operators (with non-negative kernels)
into which $I^{\a,\s}$ splits naturally. These operators are interesting
in their own right, so for the sake of completeness the lemmas provide slightly more information than
actually needed to conclude Theorem \ref{thm:LpLqlag}.

We split $I^{\a,\s}$ according to the kernel splitting
\begin{align*}
K^{\a,\s}(x,y) & = \chi_{\{x\le 2,y\le 2\}} K^{\a,\s}(x,y) + \chi_{\{x \vee y > 2\}}K^{\a,\s}(x,y) \\
	& \equiv K^{\a,\s}_0(x,y) + K^{\a,\s}_\infty(x,y)
\end{align*}
and denote the resulting integral operators by $I^{\a,\s}_0$ and $I^{\a,\s}_\infty$, respectively.
\begin{lem} \label{lem:lag_loc}
Let $\a > -1$, $\s>0$ and $1 \le p,q \le \infty$. Set
$\delta:= ((-1/2) \vee \a) + 1$. Then $I_0^{\a,\s}$ is bounded from $L^p(d\mu_{\a})$ to $L^q(d\mu_{\a})$
if and only if
$$
\frac{1}{p} - \frac{\s}{\delta} \le \frac{1}q
$$
and $(\frac{\s}{\delta},0) \neq (\frac{1}p,\frac{1}q) \neq (1,1-\frac{\s}{\delta})$.
\end{lem}
\begin{lem} \label{lem:lag_glob}
Let $\a > -1$, $\s>0$ and $1 \le p,q \le \infty$. Set $\eta :=  1/2\vee(-\a)$. 
Then $I^{\a,\s}_\infty$ is bounded from $L^p(d\mu_{\a})$ to $L^q(d\mu_{\a})$
if and only if
$$
\frac{1}p - \frac{\s}\eta \le \frac{1}q < \frac{1}p + \frac{\s}{\a+1}
$$
and $(2\s,0) \neq (\frac{1}p,\frac{1}q) \neq (1,1-2\s)$ when $\s \le \eta ={1}/2$.
\end{lem}

The first of these lemmas follows essentially from the recent results of Nowak and Roncal \cite{NR}
for potential operators in the setting of Jacobi expansions.

\begin{proof}[{Proof of Lemma \ref{lem:lag_loc}}]
In view of Theorem \ref{thm:conv_ker}, $K^{\a,\s}_{0}(x,y)$ satisfies the sharp estimates of 
Theorem \ref{thm:conv_ker} (i) in the square $0 < x,y \le 2$, and vanishes outside this square.
Comparing to \cite[Theorem~2.3]{NR}, we see that the behavior of $K^{\a,\s}_0(x,y)$ for $x,y \le 2$
is exactly the same as the behavior of the Jacobi potential kernel 
$\mathcal{K}_{\s}^{\a,\beta}(\theta,\varphi)$ in the Jacobi trigonometric polynomial setting on the
interval $(0,\pi)$. More precisely, for any fixed $\beta > -1$,
$$
K^{\a,\s}_0(x,y) \simeq \mathcal{K}_{\s}^{\a,\beta}(x,y), \qquad 0< x,y \le 2.
$$
Moreover, the Laguerre and Jacobi measures are comparable on the interval $(0,2]$.

It is now clear that the positive results of \cite[Theorem 2.3]{NR} for the Jacobi potential
operator are inherited by $I^{\a,\s}_0$. Choosing $\beta\le -1/2$ we conclude the mapping
properties of $I^{\a,\s}_0$ asserted in the lemma. If in addition $\beta \le \alpha$, all the
counterexamples and related arguments given in the proof of \cite[Theorem 2.3]{NR}, 
see \cite[Section 4.1]{NR}, remain valid for $I^{\a,\s}_0$. Hence in this case $I^{\a,\s}_0$
inherits also the negative results stated in \cite[Theorem 2.3]{NR}.
This completes the proof.
\end{proof}

To prove Lemma \ref{lem:lag_glob} we will need the technical result stated below.
\begin{lem} \label{lem:ass}
Let $\a > -1$ and $\s>0$. Then the estimates
\begin{equation}  \label{aa}
\|K^{\a,\s}_\infty(x,\cdot)\|_{L^p(d\mu_{\a})}\simeq (1\vee x)^{-2\s+2\a(1\slash p-1)}, \qquad x>0,
\end{equation}
hold for $1 \le p \le \infty$ when $\s>1/ 2$ and for $1 \le p < \frac{1}{1-2\s}$ when $\s\le 1/2$. 

Moreover, for $\sigma\le 1/2$ and $\frac{1}{1-2\s} \le p \le \infty$, we have
\begin{equation}  \label{bb}
\|K^{\a,\s}_\infty(x,\cdot)\|_{L^p(d\mu_{\a})} = \infty, \qquad x > 4.
\end{equation}
\end{lem}
Actually, only \eqref{aa} will be used in the sequel.  However, we include also \eqref{bb} to show that
\eqref{aa} is optimal in the sense of the range of admissible parameters.

\begin{proof}[{Proof of Lemma \ref{lem:ass}}]
By Theorem \ref{thm:conv_ker}, $K^{\a,\s}_\infty(x,y)$ satisfies the estimates of 
Theorem \ref{thm:conv_ker} (ii) outside the square $0 < x,y \le 2$, and vanishes inside this square.
Therefore, it is convenient to consider separately the cases $\s<1\slash2$, $\s=1\slash2$ and $\s>1\slash2$. 
In what follows we treat the case $\s<1\slash2$ leaving a similar analysis for the remaining cases to the
reader.
We only  mention that in the case $\s=1\slash2$ it is convenient to split further the kernel according to
the summands in the factor $1+\log^+ \frac{1}{|x-y|(x+y)}$. Then the part related to the $1$ 
can be included into the discussion of the case $\s<1\slash2$ to give \eqref{aa}, while the part 
coming from the $\log^+$ does not make worse the upper bound in \eqref{aa} and is decisive for \eqref{bb}. 
Finally, we observe that considering $0< x< 1$ and $x>4$ is enough for the proof of \eqref{aa} 
since for $1 \le p < \infty$ each of the two functions
$$
x \mapsto \int_a^\infty f^{\s,\a}(x,y)^py^{2\a+1}\,dy, \qquad a=0,2,
$$
where $f^{\s,\a}(x,y)$ denotes the expression on the right-hand side of ``$\simeq\simeq$'' 
in Theorem \ref{thm:conv_ker} (ii), is continuous on $(0,\infty)$; as for $p=\infty$,
the same is true for $x \mapsto \sup_{y>a}f^{\s,\a}(x,y)$, $a=0,2$, provided that $\s>1/2$. 
This may be checked in detail with the aid of the dominated convergence theorem when $p < \infty$,
or directly otherwise. 

Let $\s<1\slash2$. In view of Theorem \ref{thm:conv_ker} (ii),
$$
K_{\infty}^{\a,\s}(x,y) \simeq\simeq \chi_{\{x \vee y > 2\}} (x+y)^{-2\a-1}|x-y|^{2\s-1}
		\exp\big( -c|x-y|(x+y)\big).
$$
Therefore, if $0<x<1$, then
$$
\int_2^\infty K^{\a,\s}_\infty(x,y)^p\,y^{2\a+1}dy \simeq \simeq 
\int_2^\infty y^{(2\a+1)(1-p)+(2\s-1)p}\exp\big(-cpy^2\big)\,dy \simeq 1
$$
for $p < \infty$, and
$$
\sup_{y>2} K^{\a,\s}_\infty(x,y)\simeq \simeq \sup_{y>2} y^{2\s-2\a-2}\exp(-cy^2) \simeq 1.
$$
Thus \eqref{aa} for $x<1$ follows.
If $x>4$, then for $p < \frac{1}{1-2\s}$ and for the decisive interval $(x\slash2,3x\slash2)$ we have
\begin{align*}
\int_{x\slash2}^{3x\slash2} K^{\a,\s}_\infty(x,y)^p\,y^{2\a+1}dy &
\simeq \simeq x^{(2\a+1)(1-p)}\int_{x\slash2}^{3x\slash2} \exp(-cpx|x-y|)|x-y|^{(2\s-1)p}\,dy \\
&= 2  x^{-2\s p+2\a(1-p)}\int_{0}^{x^2\slash2} \exp(-cpu)u^{(2\s-1)p}\,du\\
&\simeq  x^{-2\s p+2\a(1-p)}.
\end{align*}
Notice that the assumption imposed on $p$ guarantees convergence of the last integral.
Checking that the relevant integrals over $(0,x/2)$ and $(3x/2,\infty)$ are controlled by
$x^{-2\s p + 2\a(1-p)}$ is straightforward. Now \eqref{aa} follows.

If $\frac{1}{1-2\s} \le p < \infty$, then the above argument leads also to \eqref{bb}. 
Finally, we have
\begin{align*}
\|K^{\a,\s}_\infty(x,\cdot)\|_{\infty} & \ge \essup_{x\slash2<y<3x\slash2} K^{\a,\s}_\infty(x,y) \\
& \simeq \simeq x^{-2\a-1}\essup_{x\slash2<y<3x\slash2} \exp(-cx|x-y|)|x-y|^{2\s-1}=\infty
\end{align*}
which justifies \eqref{bb} for $p=\infty$. 
\end{proof}

\begin{proof}[{Proof of Lemma \ref{lem:lag_glob}}] 
The structure of the proof is as follows. 
The upper estimate of Lemma~\ref{lem:ass} readily enables us to establish $L^p-L^1$ and $L^1-L^q$ 
boundedness of ${I}^{\a,\s}_\infty$ for the admissible $p$ and $q$. 
This, together with a duality argument based on the symmetry of the kernel,
$K^{\a,\s}_\infty(x,y)=K^{\a,\s}_\infty(y,x)$, and the Riesz-Thorin interpolation theorem, gives 
$L^p-L^q$ bounds for $p$ and $q$ satisfying
$$
\frac{1}p-\frac{\s}{\eta} < \frac{1}q < \frac{1}p + \frac{\s}{\a+1},
$$
where the first inequality should be replaced by a weak one in case $\eta > 1/2$.
The case when $\s < \eta ={1}/2$ and $\frac{1}p-\frac{\s}{\eta}=\frac{1}q$, $2\s < \frac{1}p < 1$,
is more subtle and will be treated by different methods. Finally, the lack of $L^p-L^q$ boundedness
for the relevant $p$ and $q$ will be shown by giving explicit counterexamples.
To simplify the notation,
in what follows $\|\cdot\|_p$ denotes the norm in the Lebesgue space $L^p(\mathbb{R}_+,d\mu_{\a})$.

The $L^p-L^1$ boundedness of ${I}^{\a,\s}_\infty$ holds for
$$
p\in 
	\begin{cases}
		[1,\infty], & \s>\a+1,\\
		[1,\frac{\a+1}{\a+1-\s}), & \s\le\a+1.
	\end{cases}
$$
Indeed, by H\"older's inequality we have
$$
\|{I}^{\a,\s}_\infty f\|_1\le \|f\|_p\big\|\|K^{\a,\s}_\infty(\cdot,y)\|_1\big\|_{p'}
$$
(here and elsewhere we use the convention that the outer norms are taken with respect to the $y$ variable)
and the assertion follows provided that $\big\|\|K^{\a,\s}_\infty(\cdot,y)\|_1\big\|_{p'}<\infty$. 
If $p=1$ this is the case for any $\s>0$ since
$$
\big\|\|K^{\a,\s}_\infty(\cdot,y)\|_1\big\|_{\infty} =\essup_{y>0}\|K^{\a,\s}_\infty(\cdot,y)\|_1 
\lesssim \sup_{y>0}(1\vee y)^{-2\s}<\infty.
$$
Similarly, for $1<p\le\infty$,
$$
\big\|\|K^{\a,\s}_\infty(\cdot,y)\|_1\big\|^{p'}_{p'}
\lesssim \int_0^\infty(1\vee y)^{-2\s p'}y^{2\a+1}\,dy<\infty,
$$
provided that $-2\s p'+2\a+1<-1$, and this happens if $p$ satisfies the imposed restrictions.

The $L^1-L^q$ boundedness of ${I}^{\a,\s}_\infty$ holds for
$$
q\in 
	\begin{cases}
		[1,\infty], & \s>1/2,\\
		[1,\frac1{1-2\s}), & \s\le1/2,
	\end{cases}
	\qquad {\rm or} \qquad
	q\in 
	\begin{cases}
		[1,\infty], & \s\ge-\a,\\
		[1,\frac{\a}{\a+\s}], & \s<-\a,
	\end{cases}
$$
when $\a\ge-1/2$ or $-1<\a<-1/2$, respectively. Indeed, by Minkowski's integral inequality 
(naturally extended to the case $q=\infty$), we get
$$
\|{I}^{\a,\s}_\infty f\|_q\le \|f\|_1\big\|\|K^{\a,\s}_\infty(\cdot,y)\|_q\big\|_{\infty}
$$
and the assertion follows provided that $\big\|\|K^{\a,\s}_\infty(\cdot,y)\|_q\big\|_{\infty}<\infty$. 
For $q=\infty$ this is the case if either $\a\ge-1/2$ and $\s>1/2$,  
or $-1<\a<-1/2$ and $\s\ge-\a$, since then
$$
\big\|\|K^{\a,\s}_\infty(\cdot,y)\|_\infty\big\|_{\infty}
\lesssim \sup_{y>0}(1\vee y)^{-2(\s+\a)}<\infty.
$$
On the other hand, for $1\le q<\infty$ in case $\s>1/2$, or for $1\le q<\frac{1}{1-2\s}$ in case $\s\le1/2$ 
(so that Lemma \ref{lem:ass} can be applied),
$$
\big\|\|K^{\a,\s}_\infty(\cdot,y)\|_q\big\|_{\infty}
\lesssim \sup_{y>0}(1\vee y)^{-2\s+2\a(1/q-1)}<\infty,
$$
provided that $\a(\frac{1}q-1)\le\s$, and this happens if $q$ satisfies the imposed restrictions. 

We now use the fact that, due to the symmetry of the kernel and a duality argument, 
$L^p-L^q$ boundedness of $I^{\a,\s}_\infty$ for some $1 \le p,q< \infty$ implies
$L^{q'}-L^{p'}$ boundedness of $I^{\a,\s}_\infty$. This allows us to infer from the results 
already obtained that  ${I}^{\a,\s}_\infty$ is $L^\infty-L^q$ bounded provided that
$$
q\in 
	\begin{cases}
		(1,\infty], & \s>\a+1,\\
		(\frac{\a+1}\s,\infty], & \s\le\a+1,
	\end{cases}
$$
and $L^p-L^\infty$ bounded provided that
$$
p\in 
	\begin{cases}
		(1,\infty], & \s>1/2,\\
		(\frac1{2\s},\infty], & \s\le1/2,
	\end{cases}	
	\qquad {\rm or} \qquad
	p\in 
	\begin{cases}
		(1,\infty], & \s\ge-\a,\\
		[-\frac{\a}{\s},\infty], & \s<-\a,
	\end{cases}
$$
when $\a\ge-1/2$ or $-1<\a<-1/2$, respectively. Using the Riesz-Thorin interpolation theorem
we conclude $L^p-L^q$ boundedness of $I^{\a,\s}_{\infty}$ in all the relevant cases, except for the one
when
\begin{equation} \label{ccc}
\s < \eta =\frac{1}{2} \quad \textrm{and} \quad \frac{1}p-\frac{\s}{\eta}=\frac{1}q \quad \textrm{and} \quad
	2\s < \frac{1}p < 1.
\end{equation}

To finish proving positive results of the lemma, we consider $\s,\a, p$ and $q$ satisfying \eqref{ccc};
in particular, now $\a \ge -1/2$. We claim that $I_{\infty}^{\a,\s}$ is $L^p-L^q$ bounded. Observe that
$$
K_{\infty}^{\a,\s}(x,y) \lesssim 
 \chi_{\{x \le 1, y> 2\}}y^{2\s-2(\a+1)}+ \chi_{\{x>2, y \le 1\}} x^{2\s-2(\a+1)}
 + \chi_{\{x,y > 1\}} (x+y)^{-2\a-1}|x-y|^{2\s-1}.
$$
By means of H\"older's inequality, it is straightforward to check that the first two terms here 
define $L^p-L^q$ bounded integral operators.
Thus our task reduces to showing that the integral operator
$$
U^{\a,\s}f(x) = \int_1^{\infty} (x+y)^{-2\a-1} |x-y|^{2\s-1} f(y)\, d\mu_{\a}(y), \qquad x > 1,
$$
satisfies the desired mapping property with respect to the measure space $((1,\infty),d\mu_{\a})$.
Since $\mu_{\a}(B(x,r)) \simeq r (x+r)^{2\a+1}$, $r>0$, $x>1$ (see \cite[Proposition 3.2]{NoSt0}; 
here the balls $B(x,r)$ are understood in the sense of the space of homogeneous type 
$((1,\infty),|\cdot|,d\mu_{\a})$), we have
\begin{equation} \label{t9}
U^{\a,\s}f(x) \simeq \int_1^{\infty} \frac{|x-y|^{2\s}}{\mu_{\a}(B(x,|x-y|))} f(y)\, d\mu_{\a}(y),
	\qquad x > 1.
\end{equation}
Integral operators of this form were investigated in \cite{AM}, among others. In particular,
taking into account that the estimate $\mu_{\a}(B(x,r)) \gtrsim r$ holds uniformly in $r>0$ and $x>1$,
we can apply \cite[Corollary 5.2]{AM} (specified to $n=1$ and $w\equiv 1$) to the operator defined
by the right-hand side in \eqref{t9}. This gives the desired conclusion for $U^{\a,\s}$.

Passing to the negative results, we must prove the following three items.
\begin{itemize}
\item[(a)]
${I}^{\a,\s}_\infty$ is not $L^p-L^q$ bounded when $\frac{1}p +\frac{\s}{\a+1}\le\frac{1}q$ and $\s\le\a+1$.
\item[(b)] 
${I}^{\a,\s}_\infty$ is not $L^p-L^q$ bounded when $\frac{1}q<\frac{1}p -\frac{\s}{\eta}$ and $\s<\eta$.
\item[(c)]  ${I}^{\a,\s}_\infty$ is not $L^p-L^q$ bounded for $(\frac{1}p,\frac{1}q)=(2\s,0)$ 
and $(\frac{1}p,\frac{1}q)=(1,1-2\s)$ when $\s\le\eta=1/2$.
\end{itemize}
To show (a), 
consider first $p=\infty$. If $\frac{\s}{\a+1} \le \frac{1}q$ then, by Lemma \ref{lem:ass},
$$
\|I_{\infty}^{\a,\s}\boldsymbol{1}\|_q^q \simeq \int_0^{\infty} (1\vee x)^{-2\s q} x^{2\a+1}\, dx = \infty,
$$
hence ${I}^{\a,\s}_\infty$ is not $L^\infty-L^q$ bounded. 
To treat the case $p< \infty$, we may assume in addition that $\frac{1}q = \frac{1}p + \frac{\s}{\a+1}$,
because of an interpolation argument.
Let $f(y)=\chi_{\{y>e\}}y^{-2(\a+1)/p}(\log y)^{-1/p-\s/(\a+1)}$. Then
$$
\|f\|_p^p = \int_e^\infty (\log y)^{-1-\s p/(\a+1)}\frac{dy}y < \infty,
$$
so $f\in L^p(d\mu_\a)$. We claim that ${I}^{\a,\s}_\infty f\notin L^q(d\mu_\a)$.
Indeed, using the lower bound from Theorem~\ref{thm:conv_ker}~(ii), for $x>2e$ we obtain
\begin{align*}
I_{\infty}^{\a,\s}f(x) & \ge f(x) \int_{x/2}^x K_{\infty}^{\a,\s}(x,y)\, d\mu_{\a}(y) \\
& \gtrsim f(x) \int_{x/2}^x \left\{ 
	\begin{array}{ll}
		(x-y)^{2\s-1}, & \s < 1/2 \\
		1+\log^+\frac{1}{x(x-y)}, & \s=1/2 \\
		x^{1-2\s}, & \s > 1/2
	\end{array}
	\right\}
	\exp\big(-cx (x-y)\big) \, dy \\
& \simeq x^{-2\s}f(x),	
\end{align*}
where the last relation follows by the change of variable $y=x-u/x$. Consequently,
$$
{I}^{\a,\s}_\infty f(x)\gtrsim  x^{-2(\a+1)/p-2\s}(\log x)^{-1/p-\s/(\a+1)}
=x^{-2(\a+1)/q}(\log x)^{-1/q},\qquad x>2e,
$$
and the claim follows. 

To justify (b), we fix $p$ and $q$ satisfying the assumed conditions and 
first consider the case $\a \ge -1/2$. This means that $\eta = 1/2$ and $\s < 1/2$.
Then, by Theorem \ref{thm:conv_ker} (ii),
\begin{equation} \label{a5}
K_{\infty}^{\a,\s}(x,y) \gtrsim |x-y|^{2\s-1}, \qquad x,y \in (2,4).
\end{equation}
Let $f(y) = \chi_{(2,3)}(y)\, (3-y)^A$, where we take $A = -\frac{1}p + \varepsilon$, with $\varepsilon$
satisfying $0 < \varepsilon < \frac{1}p-2\s-\frac{1}q$. Clearly, $f \in L^p(d\mu_{\a})$ since
$$
\|f\|_p^p \simeq \int_2^3 (3-y)^{-1+\varepsilon p}\, dy < \infty.
$$
We will show that $I_{\infty}^{\a,\s}f \notin L^q(d\mu_{\a})$. Changing the variable of integration, we get
$$
I_{\infty}^{\a,\s}f(x) \gtrsim \int_2^3 (x-y)^{2\s-1} (3-y)^A\, dy = 
	(x-3)^{2\s+A} \int_0^{1/(x-3)} \frac{u^A\, du}{(1+u)^{1-2\s}}, \qquad x \in (3,4).
$$
Here $1/(x-3)>1$, so the last integral is larger than a positive constant. Thus
$$
I_{\infty}^{\a,\s}f(x) \gtrsim (x-3)^{2\s+A}, \qquad x \in (3,4).
$$
Since $2\s+A<0$, we see that $I_{\infty}^{\a,\s}f$ is not in $L^{\infty}$. Neither it belongs to
$L^q(d\mu_{\a})$ when $q< \infty$, because $(2\s+A)q<-1$ and, consequently,
$$
\|I_{\infty}^{\a,\s}f\|_q^q \gtrsim \int_3^4 (x-3)^{(2\s+A)q}\, dx = \infty.
$$

The case $\a < -1/2$ is slightly more subtle. Now $\eta = -\a$ and we must show that $I_{\infty}^{\a,\s}$
is not $L^p-L^q$ bounded whenever $\frac{1}q < \frac{1}p + \frac{\s}{\a}$. Let $n$ be large. Observe that 
in view of Theorem~\ref{thm:conv_ker}~(ii), for $x,y\in (n,n+1/n)$ we have
$$
K_{\infty}^{\a,\s}(x,y) \gtrsim n^{-2\a-1} \begin{cases}
		|x-y|^{2\s-1}, & \s < 1/2,\\
		\log\frac{2}{n|x-y|}, & \s=1/2,\\
		n^{1-2\s}, & \s > 1/2,
	\end{cases}
$$
uniformly in $x,y$ and $n$. Take $f_n = \chi_{(n,n+1/n)}$. Then $\|f_n\|_p \simeq n^{2\a/p}$. Further,
assume that $x \in (n,n+1/n)$. If $\s > 1/2$, then
$$
I_{\infty}^{\a,\s}f_n(x) \gtrsim n^{-2\a-2\s}\int_n^{n+1/n} y^{2\a+1}\, dy \simeq n^{-2\s}.
$$
If $\s < 1/2$, then
$$
I_{\infty}^{\a,\s}f_n(x) \gtrsim
	\int_n^{n+1/n} |x-y|^{2\s-1}\, dy \simeq \int_n^{n+1/n} (y-n)^{2\s-1}\, dy \simeq n^{-2\s}.
$$
For $\s=1/2$ we also have
$$
I_{\infty}^{\a,\s}f_n(x) \gtrsim \int_n^{n+1/n} \log\frac{2}{n|x-y|} \, dy
	\simeq \int_n^{n+1/n} \log\frac{2}{n(y-n)}\, dy \simeq n^{-1} = n^{-2\s}.
$$
Thus, in all the cases, 
$$
\|I_{\infty}^{\a,\s}f_n\|_q^q \gtrsim \int_n^{n+1/n} n^{-2\s q} y^{2\a+1}\, dy \simeq n^{-2\s q + 2\a}.
$$
Consequently, with $q=\infty$ also admitted,
$$
\frac{\|I_{\infty}^{\a,\s}f_n\|_q}{\|f_n\|_p} \gtrsim n^{-2\s-2\a(1/p-1/q)}.
$$
Since $-2\s-2\a(\frac{1}p-\frac{1}q)>0$, the norm ratio is not bounded as $n \to \infty$.

Proving (c), we begin with the extreme case $\s=\eta=1/2$ and show that ${I}^{\a,\s}_\infty$ 
is not $L^1-L^\infty$ bounded. Let $f_n = \chi_{(3-1/n,3)}$ with $n$ large. Then
$\|f_n\|_1 \simeq n^{-1}$ and by Theorem \ref{thm:conv_ker} (ii)
$$
\|I_{\infty}^{\a,\s}f_n\|_{\infty} \gtrsim \essup_{3<x<3+1/n} \int_{3-1/n}^3\log\frac{1}{x-y}\, dy
= \int_{3-1/n}^3 \log\frac{1}{3-y}\, dy \simeq \frac{1}n \log n,
$$
and the conclusion follows by letting $n \to \infty$.

Next assume that $\s<\eta=1/2$. By duality, it is enough to check that $I_{\infty}^{\a,\s}$ is not
bounded from $L^{1/(2\s)}(d\mu_{\a})$ to $L^{\infty}$. Observe that in this situation \eqref{a5} holds.
Take $f(y) = \chi_{(2,3)}(y)/((3-y)^{2\s}\log\frac{2}{3-y})$. Then $f \in L^{1/(2\s)}(d\mu_{\a})$, but
$$
\|I_{\infty}^{\a,\s}f\|_{\infty} \gtrsim 
	\essup_{3<x<4} \int_2^3 \frac{|x-y|^{2\s-1}\, dy}{(3-y)^{2\s}\log\frac{2}{3-y}} =
	\int_2^3 \frac{dy}{(3-y)\log\frac{2}{3-y}} = \infty.
$$
This finishes the verification of item (c).

The proof of Lemma \ref{lem:lag_glob} is complete.
\end{proof}

\subsection{$\boldsymbol{L^p-L^q}$ estimates in the Laguerre setting of Hermite type}
Similarly to Theorem~\ref{thm:LpLqlag}, Theorem~\ref{thm:lag_H}
follows readily from the two lemmas below describing $L^p-L^q$ behavior
of two auxiliary operators with non-negative kernels which $I_{H}^{\a,\s}$ splits into. More precisely,
we split the operator $I^{\a,\s}_{H}$ according to the kernel splitting
\begin{align*}
K^{\a,\s}_{H}(x,y) & = 
	\chi_{\{x \le 2, y \le 2\}} K^{\a,\s}_{H}(x,y) + \chi_{\{x \vee y > 2\}}K^{\a,\s}_{H}(x,y) \\
    & \equiv K^{\a,\s}_{H,0}(x,y) + K^{\a,\s}_{H,\infty}(x,y)
\end{align*}
and denote the resulting integral operators by $I^{\a,\s}_{H,0}$ and $I^{\a,\s}_{H,\infty}$, respectively.
\begin{lem} \label{lem:lag_loc_H}
Let $\a > -1$, $\s>0$ and $1 \le p,q \le \infty$.  
\begin{itemize}
\item[(a)] If $\a \ge -1/2$, then ${I}^{\a,\s}_{H,0}$ is bounded from $L^p(dx)$ to $L^q(dx)$ if and only if
$$
\frac{1}p - 2\s\le \frac{1}q 
\quad \textrm{and} \quad
	\bigg(\frac{1}p,\frac{1}q\bigg) \notin \big\{(2\s,0), (1,1-2\s)\big\}.
$$
\item[(b)] Let $\a < -1/2$. Then $L^p(dx)\subset {\rm Dom}\, I_{H,0}^{\a,\s}$ if and only if
$p > 2/(2\a+3)$. In this case
${I}^{\a,\s}_{H,0}$ is bounded from $L^p(dx)$ to $L^q(dx)$ if and only if
$$
\frac{1}p - 2\s \le \frac{1}q  \quad  \textrm{and} \quad \frac{1}q > -\a-\frac{1}2.
$$
\end{itemize}
\end{lem}
\begin{lem} \label{lem:lag_glob_H}
Let $\a > -1$, $\s>0$ and $1 \le p,q \le \infty$.  Then  ${I}^{\a,\s}_{H,\infty}$ 
satisfies the positive $L^p-L^q$ mapping properties stated in Theorem \ref{thm:lag_H} for $I_{H}^{\a,\s}$.

On the other hand, $I_{H}^{\a,\s}$ is not bounded from $L^p(dx)$ to $L^q(dx)$ when 
$$
\frac{1}q \ge \frac{1}p+2\s.
$$
\end{lem}

The proof of Lemma \ref{lem:lag_loc_H} uses the sharp description of $L^p-L^q$ boundedness for
the potential operator in the Jacobi trigonometric `function' setting stated in \cite[Theorem 2.4]{NR}
and arguments analogous to those from the proof of Lemma \ref{lem:lag_loc}; 
the first part of Lemma \ref{lem:lag_loc_H} (b) may be verified directly. We omit the details.
To prove Lemma \ref{lem:lag_glob_H} we will mostly appeal to the results obtained in the setting of Laguerre
expansions of convolution type. Essentially, only the case $\a < -1/2$ requires new arguments.
However, we first give an analogue of Lemma \ref{lem:ass}. Although we will use only a part of it, we provide
a full statement for the sake of completeness and, perhaps, reader's curiosity.

\begin{lem} \label{lem:ass_H}
Let $\a > -1$, $\s>0$ and $1 \le p \le \infty$. Then the estimates
$$
\|K^{\a,\s}_{H,\infty}(x,\cdot)\|_p\simeq
\begin{cases}
        x^{\a+1/2}, & x\le1,\\
        x^{-2\s+1-1/p}, & x>1,
    \end{cases}
$$
hold provided that $p$ satisfies $\frac{1}p > 1-2\s$ and, in addition, $\frac{1}p > - \a - 1/2$ in case
$\a < -1/2$. 

Moreover, for the remaining $p$ we have
\begin{equation}  \label{smr}
\|K^{\a,\s}_{H,\infty}(x,\cdot)\|_p=\infty, \qquad x>4.
\end{equation}
\end{lem}

\begin{proof}
The reasoning relies on the arguments from the proof of Lemma \ref{lem:ass}. We will give some details
for the case $\s< 1/2$ leaving the remaining analysis to the reader.

Let $\s<1\slash2$. In view of \eqref{linkKH} and Theorem \ref{thm:conv_ker} (ii),
$$
K^{\a,\s}_{H,\infty}(x,y)\simeq \simeq \chi_{\{x \vee y >2\}} (xy)^{\a+1/2} 
(x+y)^{-2\a-1}|x-y|^{2\s-1}\exp\big(-c|x-y|(x+y)\big).
$$
Hence, for $x<1$ and $y>2$ we have
$$
K^{\a,\s}_{H,\infty}(x,y)\simeq \simeq x^{\a+1/2}y^{-\a+2\s-3/2}\exp(-cy^2),
$$
while for $x>4$ and $y>0$
\begin{equation*}
K^{\a,\s}_{H,\infty}(x,y)\simeq \simeq
\begin{cases}
        x^{-\a+2\s-3/2}y^{\a+1/2}\exp(-cx^2), & 0<y\le x/2,\\
        |x-y|^{2\s-1}\exp\big(-cx|x-y|\big), & x/2<y<3x/2,\\
        x^{\a+1/2}y^{-\a+2\s-3/2}\exp(-cy^2), & 3x/2\le y<\infty.
    \end{cases}
\end{equation*}
Therefore, if $0<x<1$, then
$$
\|K_{H,\infty}^{\a,\s}(x,\cdot)\|_p \simeq \simeq 
x^{\a+1/2} \big\|\chi_{\{y>2\}}y^{-\a+2\s-3/2}\exp(-cy^2)\big\|_{p}
\simeq x^{\a+1/2}.
$$
If $x>4$, then on the decisive interval $(x\slash2,3x\slash2)$ we have
\begin{align*}
\int_{x\slash2}^{3x\slash2} K^{\a,\s}_{H,\infty}(x,y)^p\,dy &
\simeq \simeq \int_{x\slash2}^{3x\slash2} |x-y|^{(2\s-1)p}\exp(-cpx|x-y|)\,dy \\
&= 2  x^{-(2\s-1)p-1}\int_{0}^{x^2\slash2} \exp(-cpu)u^{(2\s-1)p}\,du\\
&\simeq  x^{(1-2\s)p-1},
\end{align*}
provided that $\frac{1}p > 1-2\s$; this condition is necessary and sufficient for finiteness of the integral.
As easily verified, the relevant integrals over $(3x/2,\infty)$ and $(0,x/2)$ are controlled by
$x^{(1-2\s)p-1}$. In the latter case one has to impose the condition $\frac{1}p>-\a-1/2$ in case
$\a < -1/2$ since otherwise the integral is infinite. Checking \eqref{smr} for the $p$ in question is
straightforward.
\end{proof}

\begin{proof}[{Proof of Lemma \ref{lem:lag_glob_H}}]
In view of \eqref{linkKH} and Theorem \ref{thm:conv_ker} (ii), for $\a \ge -1/2$ the kernel
$K_{H,\infty}^{\a,\s}(x,y)$ is controlled by $K_{\infty}^{-1/2,\s}(x,y)$. Thus $I_{H,\infty}^{\a,\s}$
inherits the $L^p-L^q$ boundedness of $I_{\infty}^{-1/2,\s}$ 
(note that $d\mu_{-1/2}$ is the Lebesgue measure).
This together with Lemma \ref{lem:lag_glob} gives the positive results of the lemma in case $\a \ge -1/2$.

Next observe that for any $\a > -1$, the two above mentioned kernels are comparable if the arguments are,
see \eqref{asym},
\begin{equation} \label{krel}
K_{H,\infty}^{\a,\s}(x,y) \simeq K_{\infty}^{-1/2,\s}(x,y), \qquad x/2 < y < 2x.
\end{equation}
So to prove the required negative result in case $p<\infty$ we can use the counterexample from (a)
of the proof of Lemma \ref{lem:lag_glob}, since it involves only comparable arguments of the kernel.
In case $p=\infty$ the conclusion follows by Lemma \ref{lem:ass_H}, since we can write
$$
\|I_{H,\infty}^{\a,\s}\boldsymbol{1}\|_q^q \gtrsim \int_1^{\infty} x^{-2\s q}\, dx = \infty
$$
provided that $\frac{1}q \ge 2\s$ and, in addition, $\frac{1}q > -\a-1/2$ in case $\a< -1/2$.

It remains to justify the $L^p-L^q$ boundedness in case $\a < -1/2$. Because of \eqref{krel} and
Lemma~\ref{lem:lag_glob}, it is enough to study the mutually dual integral operators
\begin{align*}
U_1^{\a,\s}f(x) & = \int_0^{x/2} K_{H,\infty}^{\a,\s}(x,y)f(y)\, dy,\\
U_2^{\a,\s}f(x) & = \int_{2x}^{\infty} K_{H,\infty}^{\a,\s}(x,y)f(y)\, dy.
\end{align*}
Assuming that $p>1$ and $q < \infty$, we will show that $U_1^{\a,\s}$ and $U_2^{\a,\s}$ are
$L^p-L^q$ bounded when
$$
\frac{1}{p'} > -\a-1/2 \quad \textrm{and} \quad \frac{1}q > -\a-1/2 \quad \textrm{and} \quad
	\frac{1}q \ge \frac{1}p - 2\s.
$$
This will finish the proof.

By \eqref{linkKH}, Theorem \ref{thm:conv_ker} (ii) and H\"older's inequality,
$$
|U_1^{\a,\s}f(x)| \lesssim x^{-\a-1/2} \exp\big(-cx^2\big) 
	\left\{ 
	\begin{array}{ll}
		x^{2\s-1}, & \s < 1/2 \\
		1+\log^+\frac{1}{x}, & \s=1/2 \\
		(1+x)^{1-2\s}, & \s > 1/2
	\end{array}
	\right\}
	 \big\|\chi_{\{y < x\}} y^{\a+1/2}\big\|_{p'} \|f\|_p.
$$
If $(\a+1/2)p'>-1$, the $L^{p'}$ norm here is finite and comparable to $x^{\a+3/2-1/p}$. Then we get
$$
|U_1^{\a,\s}f(x)| \lesssim g^{\s}(x) \|f\|_p, \qquad x >0,
$$
where
$$
g^{\s}(x) = x^{1-1/p} \exp\big(-cx^2\big) 
	\begin{cases}
		x^{2\s-1}, & \s < 1/2, \\
		1+\log^+\frac{1}{x}, & \s=1/2, \\
		(1+x)^{1-2\s}, & \s > 1/2.
	\end{cases}
$$
It is easy to check that $g^{\s}\in L^q$ when $\s \ge 1/2$. The same is true for $\s<1/2$ under the
additional condition $\frac{1}q> \frac{1}p -2\s$. So in these cases $U_1^{\a,\s}$ is $L^p-L^q$ bounded.
For $\s<1/2$ and $\frac{1}q= \frac{1}p -2\s$ we have $g^{\s}(x) = x^{-1/q} \exp(-cx^2)$.
Since now $g^{\s}$ belongs to weak $L^q$, we see that $U_1^{\a,\s}$ is of weak type $(p,q)$.
Now the $L^p-L^q$ boundedness follows by the Marcinkiewicz interpolation theorem.

Considering $U_{2}^{\a,\s}$, we recall that it is the dual of $U_1^{\a,\s}$ and use the already proved
results for $U_{1}^{\a,\s}$. This gives the desired $L^p-L^q$ boundedness, except for the case $q=1$
which we now treat separately. By \eqref{linkKH} and Theorem \ref{thm:conv_ker} (ii),
$$
|U_2^{\a,\s}f(x)| \lesssim x^{\a+1/2} \int_x^{\infty} y^{-\a-1/2}\exp\big(-cy^2\big)
	\left\{ 
	\begin{array}{ll}
		y^{2\s-1}, & \s < 1/2 \\
		1+\log^+\frac{1}{y}, & \s=1/2 \\
		(1+y)^{1-2\s}, & \s > 1/2
	\end{array}
	\right\}
	|f(y)|\, dy.
$$
Integrating in $x$ and changing the order of integration produces
$$
\|U_2^{\a,\s}f\|_1 \lesssim \int_0^{\infty} y^{1/p} g^{\s}(y) |f(y)|\, dy.
$$
Since the function $y \mapsto y^{1/p}g^{\s}(y)$ belongs to $L^{p'}$, the conclusion follows by
H\"older's inequality. 
\end{proof}

\subsection{$\boldsymbol{L^p-L^q}$ estimates in the Dunkl-Laguerre setting}
Let us first introduce some extra notation. For a function $f$ on $\mathbb{R}$, define $f_+$ and $f_{-}$
as functions on $\mathbb{R}_+$ given by $f_{\pm}(x) = f(\pm x)$, $x>0$. In a similar way, let
$\mathcal{K}^{\a,\s}_{+}$ and $\mathcal{K}^{\a,\s}_{-}$ be the kernels on $\mathbb{R}_+\times \mathbb{R}_+$
determined by $\mathcal{K}_{\pm}^{\a,\s}(x,y)={K}_D^{\a,\s}(x,\pm y)$, $x,y > 0$. 
Denote the corresponding integral operators related to the measure space $(\mathbb{R}_+,d\mu_{\a})$ by
$\mathcal{I}_{\pm}^{\a,\s}$, respectively.

Clearly, for any fixed $1 \le p \le \infty$,
$$
\|f\|_{L^p(dw_{\a})} \simeq \|f_+\|_{L^p(d\mu_{\a})} + \|f_{-}\|_{L^p(d\mu_{\a})}.
$$
Further, by the symmetry of the kernel, ${K}_D^{\a,\s}(-x,y)={K}_D^{\a,\s}(x,-y)$, and the symmetry
of $w_{\a}$,
$$
\big(I_D^{\a,\s}f\big)_\pm=\mathcal I^{\a,\s}_+(f_\pm)+\mathcal I^{\a,\s}_-(f_\mp).
$$

\begin{proof}[{Proof of Theorem \ref{thm:lag_D}}]
In view of \eqref{kkp1}, the kernels
$\mathcal{K}^{\a,\s}_{\pm}(x,y)$ are controlled by $K^{\a,\s}(x,y)$.
Thus $\mathcal{I}^{\a,\s}_{\pm}$ satisfy
the positive mapping properties from Theorem \ref{thm:LpLqlag}. Therefore,
for the asserted $p$ and $q$ we can write
\begin{align*}
& \|I_D^{\a,\s}f\|_{L^q(dw_{\a})}\\
&\simeq \|(I_D^{\a,\s}f)_+\|_{L^q(d\mu_{\a})} + \|(I_D^{\a,\s}f)_-\|_{L^q(d\mu_{\a})} \\ 
&\le \|\mathcal I^{\a,\s}_+(f_+)\|_{L^q(d\mu_{\a})} +  
\|\mathcal I^{\a,\s}_-(f_-)\|_{L^q(d\mu_{\a})}+\|\mathcal I^{\a,\s}_+(f_-)\|_{L^q(d\mu_{\a})} 
+ \|\mathcal I^{\a,\s}_-(f_+)\|_{L^q(d\mu_{\a})}\\
&\lesssim \|f_+\|_{L^p(d\mu_{\a})} +  \|f_-\|_{L^p(d\mu_{\a})}\\
&\simeq \|f\|_{L^p(dw_{\a})}.
\end{align*}

To show the necessity part, we observe that by \eqref{kkp2} the kernel 
$\mathcal{K}^{\a,\s}_{+}(x,y)$ is comparable to $K^{\a,\s}(x,y)$.
Thus the range of admissible $p$ and $q$ from 
Theorem \ref{thm:LpLqlag} is optimal also for $\mathcal{I}_+^{\a,\s}$.
Now, to finish the proof it suffices to notice that
if ${I}_D^{\a,\s}$ is $L^p-L^q$ bounded, then so is $\mathcal{I}^{\a,\s}_+$. Indeed, take a function
$f$ on $\mathbb{R}_+$ and extend it to $\tilde{f}$ on $\mathbb{R}$ by setting $\tilde{f}(x)=0$ for
$x \notin \mathbb{R}_+$. Then, assuming that ${I}_D^{\a,\s}$ is $L^p-L^q$ bounded,
$$
\|\mathcal{I}_+^{\a,\s}f\|_{L^q(d\mu_{\a})} = \|({I}_D^{\a,\s}\tilde{f})_+\|_{L^q(d\mu_{\a})}
\le \|{I}_D^{\a,\s}\tilde{f}\|_{L^q(dw_{\a})} \lesssim \|\tilde{f}\|_{L^p(dw_{\a})}
= \|f\|_{L^p(d\mu_{\a})}.
$$
The conclusion follows.
\end{proof}



\begin{thebibliography}{10}

\bibitem{AM}
P.\ Auscher, J.M.\ Martell,
\emph{Weighted norm inequalities for fractional operators},
Indiana Univ{.}\ Math{.}\ J{.}\ 57 (2008), 1845--1869.

\bibitem{BoTo1} 
B{.}\ Bongioanni, J{.}L{.}\ Torrea, 
\emph{Sobolev spaces associated to the harmonic oscillator}, 
Proc{.}\ Indian Acad{.}\ Sci{.}\ Math{.}\ Sci{.}\ 116 (2006), 337--360. 

\bibitem{CiRo}
\'O.\ Ciaurri, L.\ Roncal,
\emph{Vector-valued extensions for fractional integrals of Laguerre expansions},
preprint 2012. \texttt{arXiv:1212.4715}

\bibitem{Leb}
N{.}N{.}\ Lebedev,
\emph{Special functions and their applications},
Dover, New York, 1972.

\bibitem{Na}
I{.}\ N{\aa}sell,
\emph{Rational bounds for ratios of modified Bessel functions},
SIAM J{.}\ Math{.}\ Anal{.}\ 9 (1978), 1--11.

\bibitem{NR}
A{.}\ Nowak, L{.}\ Roncal,
\emph{Potential operators associated with Jacobi and Fourier-Bessel expansions}, 
preprint 2012. \texttt{arXiv:1212.6342}

\bibitem{NoSj}
A{.}\ Nowak, P{.}\ Sj\"ogren,
\emph{The multi-dimensional pencil phenomenon for Laguerre heat-diffusion maximal operators},
Math{.}\ Ann{.}\ 344 (2009), 213--248.

\bibitem{NoSt0} 
A{.}\ Nowak, K{.}\ Stempak, 
\emph{Riesz transforms for multi-dimensional Laguerre function expansions}, 
Adv{.}\ Math{.}\ 215 (2007), 642--678.

\bibitem{NoSt1}
A.\ Nowak, K.\ Stempak,
\emph{Negative powers of Laguerre operators},
Canad.\ J.\ Math.\ 64 (2012), 183--216.

\bibitem{NoSt2}
A.\ Nowak, K.\ Stempak,
\emph{Sharp estimates of the potential kernel for the harmonic oscillator with applications},
Nagoya Math.\ J.\ 212 (2013), 1--17.

\bibitem{NoSz}
A.\ Nowak, T.Z.\ Szarek,
\emph{Calder\'on-Zygmund operators related to Laguerre function expansions of convolution type}, 
J.\ Math.\ Anal.\ Appl.\ 388 (2012), 801--816.

\bibitem{St}
K.\ Stempak,
\emph{Almost everywhere summability of Laguerre series},
Studia Math.\ 100 (1991), 129--147.

\bibitem{Th}
S{.}\ Thangavelu,
\emph{Lectures on Hermite and Laguerre expansions},
Mathematical Notes 42, Princeton University Press, Princeton, 1993.

\end{thebibliography}
\end{document}